\documentclass[reqno,11pt]{amsart} 

\usepackage{stix}
\usepackage{amssymb}
\usepackage[top=1.5in,right=1.125in,left=1.125in,bottom=1.5in]{geometry}
\usepackage{color}
\definecolor{red}{rgb}{0.7,0,0}
\definecolor{grey}{RGB}{112,112,112}
\definecolor{blue}{RGB}{034,113,179}
\usepackage[colorlinks=true,citecolor=blue,linkcolor=red,urlcolor=grey,backref=page]{hyperref}

\usepackage{comment}

\newtheorem{theo}{Theorem}[section] 
\newtheorem{prop}[theo]{Proposition}

\newtheorem{col}[theo]{Corollary}

\theoremstyle{remark}
\newtheorem{rmk}[theo]{Remark}

\newcounter{mnotecount}[section]

\renewcommand{\themnotecount}{\thesection.\arabic{mnotecount}}

\newcommand{\mnote}[1]
{\protect{\stepcounter{mnotecount}}$^{\mbox{\footnotesize
$
\bullet$\themnotecount}}$ \marginpar{
\raggedright\tiny\em
$\!\!\!\!\!\!\,\bullet$\themnotecount: #1} }

\newcommand{\hook}{{\setlength{\unitlength}{11pt}   
                   \begin{picture}(.833,.8)
                   \put(.15,.08){\line(1,0){.35}}
                   \put(.5,.08){\line(0,1){.5}}
                   \end{picture}}}
\newcommand{\CP}{\mathbb{CP}}
\newcommand{\C}{\mathbb{C}}
\newcommand{\PP}{\mathbb{P}}
\newcommand{\RP}{\mathbb{RP}}
\newcommand{\R}{\mathbb{R}}
\newcommand{\eee}{\mathcal{V}}
\newcommand{\projb}{P_{[\nabla]}}
\newcommand{\bilb}{B_{(\Omega,I)}}
\newcommand{\bX}{X}
\newcommand{\Rho}{\mathrm{P}}
\newcommand{\mf}{N}
\newcommand{\bw}{w}
\newcommand{\mm}{\setminus}

\def\p{\partial}
\def\be{\begin{equation}}
\renewcommand{\d}{d}
\newcommand{\bv}{v}

\def\ee{\end{equation}}
\def\ep{{\varepsilon}}
\def\bea{\begin{eqnarray}}
\def\eea{\end{eqnarray}}

\newcommand{\spp}{\mathbb{S}}

\newcommand{\tr}{\operatorname{tr}}

\numberwithin{equation}{section}
\begin{document} \date{June 6, 2017}
\title[Projective surfaces and anti-self-dual Einstein metrics]{Gauge theory on projective surfaces and anti-self-dual Einstein metrics in dimension four}

\author[M.~Dunajski]{Maciej Dunajski}
\address{Department of Applied Mathematics and Theoretical Physics\\ 
University of Cambridge\\ Wilberforce Road, Cambridge CB3 0WA, UK.}
\email{m.dunajski@damtp.cam.ac.uk}
\author[T.~Mettler]{Thomas Mettler}
\address{Institut f\"ur Mathematik\\
Goethe-Universit\"at Frankfurt\\
Robert-Mayer-Str.~10, 60325 Frankfurt, Germany.
}
\email{mettler@math.uni-frankfurt.de}
\begin{abstract}
Given a projective structure  on a surface $\mf$, we show how to canonically construct a neutral signature Einstein metric with non-zero scalar curvature as well as a symplectic form on the total space $M$ of a certain rank $2$ affine bundle $M \to \mf$. The Einstein metric has anti-self-dual conformal curvature and admits a parallel field of anti-self-dual planes. We show that locally every such metric arises from our construction unless it is conformally flat. The homogeneous Einstein metric corresponding to the flat projective structure on $\RP^2$ is the non-compact real form of the Fubini-Study metric on  $M=\mathrm{SL}(3, \R)/\mathrm{GL}(2, \R)$. We also show how our construction relates to a certain gauge-theoretic equation introduced by Calderbank.
\end{abstract}
\maketitle
\section{Introduction}

A projective structure on a smooth surface $N$ is an equivalence class $[\nabla]$ of torsion-free connections on $TN$ having the same unparametrised geodesics. Canonically associated to a projective surface $(N,[\nabla])$ is a rank $2$ affine bundle $M \to N$ which is modelled on $T^*N$ and which arises as the complement of a certain $\RP^1$-subbundle of the projectivised cotractor bundle $\PP(E)\to N$ of $(N,[\nabla])$. The aim of this paper is to canonically construct a pair $(g,\Omega)$ on $M$, consisting of a neutral signature anti-self-dual (ASD) Einstein metric $g$, as well as a symplectic form $\Omega$. The pair $(g,\Omega)$ is related by an endomorphism $I : TM \to TM$ whose square is the identity and hence it defines what is known as a bi-Lagrangian structure or almost para-K\"ahler structure on $M$. We construct the pair $(g,\Omega)$ by taking a $\mathrm{GL}(2,\R$)-quotient of the Cartan geometry associated to $(N,[\nabla])$ and in doing so, establish a one-to-one correspondence between projective vector fields on $(N,[\nabla])$ and sympletic Killing vector fields on $(M,g,\Omega)$. In addition, we observe that every Killing vector field of $(M,g)$ is symplectic with respect to $\Omega$ and hence the lift of a projective vector field on $(N,[\nabla])$. 

The sections of the affine bundle $M \to N$ are in one-to-one correspondence with the $[\nabla]$ representative connections and hence the choice of a representative connection $\nabla \in [\nabla]$ provides a diffeomorphism $T^*N \to M$. Pulling back the pair $(g,\Omega)$ with this diffeomorphism gives a pair $(g_{\nabla},\Omega_{\nabla})$ on $T^*N$ which -- in canonical local coordinates $(x^i,\xi_i)$ on the contangent bundle $\nu : T^*N \to N$ -- takes the form
\begin{align*}
g_{\nabla}&=\left(d\xi_i-\left(\Gamma_{ij}^k \xi_k-\xi_i\xi_j-R_{(ij)}\right)\d x^j\right)\odot \d x^i,\\
\Omega_{\nabla}&=\d \xi_i\wedge \d x^i+\frac{1}{3}R_{[ij]}\d x^i\wedge \d x^j. 
\end{align*}
Here $\Gamma^i_{jk}$ denote the Christoffel symbols and $R_{(ij)},R_{[ij]}$ the components of the symmetric -- and anti-symmetric part of the Ricci curvature of $\nabla$ with respect to the coordinates $(x^i)$. The first two summands in the coordinate expression for the metric $g_{\nabla}$ give the classical Patterson-Walker metric $h_{\nabla}$ which is canonically defined on $T^*N$ from a torsion-free connection $\nabla$ on $N$. The metric $g_{\nabla}$ is thus part of a one-parameter family $g_{\nabla,\Lambda}$ of metrics on $T^*N$ defined by
\begin{equation}\label{main_metric}
g_{\nabla,\Lambda}=h_{\nabla}+\Lambda \lambda^2+\left(\frac{1}{\Lambda}\right)\nu^*\mathrm{Ric}^+(\nabla),
\end{equation}
where $\lambda$ denotes the tautological $1$-form of $T^*N$ and $\Lambda$ is any non-zero real number. The family of metrics $g_{\nabla,\Lambda}$ already appeared in~\cite{spanish} where they are locally characterized as the neutral signature type II Osserman metrics whose Jacobi operator have non-zero eigenvalues. However, the relation of the metric $g_{\nabla}=g_{\nabla,1}$ to projective differential geometry is not noted there. The reader may also consult~\cite{derdzinski2,MR2211331} and references cited therein for results about the classification of neutral signature four-dimensional Osserman metrics. All the metrics in the family $g_{\nabla,\Lambda}$ are anti-self-dual and Einstein with scalar curvature $24\Lambda\neq 0$. Moreover, in Theorem \ref{theo_asd} we show that all ASD Einstein metrics which admit a parallel ASD totally null distributions are locally of the form (\ref{main_metric}). We also observe that if a connection $\nabla$ has skew-symmetric Ricci tensor, then the limit $\Lambda \to 0$ of the above family $g_{\nabla,\Lambda}$ yields an anti-self-dual Ricci flat metric which previously appeared in the work of Derdzi\'nski \cite{derdzinski}. 

In the final part of the article we relate the metric $g$ to a certain gauge-theoretic equation introduced by Calderbank in~\cite{Cal1}. We also discuss some examples. 

This paper mainly concerns itself with the two-dimensional case, but there are obvious higher dimensional generalisations which we briefly discuss in an Appendix.  

\subsection*{Acknowledgments} The authors wish to thank
Andreas \v{C}ap, Andrzej Derdzi\'nski, Nigel Hitchin, and Claude LeBrun for helpful discussions regarding the contents of this paper. 
TM is grateful for travel support via the grant SNF 200020\_144438 of the Swiss National Science Foundation. 

\section{Preliminaries}

\subsection{Algebraic preliminaries}
As usual, we let $\mathbb{R}^n$ the denote the space of column vectors of height $n$ with real entries and $\mathbb{R}_n$ the space of row vectors of length $n$ with real entries. Matrix multiplication $\mathbb{R}_n \times \mathbb{R}^n \to \mathbb{R}$ is a non-degenerate pairing identifying $\R_n$ with the dual vector space of $\R^n$. 

Let $\mathbb{RP}^2=(\R^3\setminus\{0\})/\R^*$ denote space of lines in $\R^3$ through the origin, i.e., two-dimensional real projective space. For any nonzero $x \in \R^3$ let $[x]$ denote its corresponding point in $\mathbb{RP}^2$. Let $\mathbb{RP}_2=(\R_3\setminus\{0\})/\R^*$ denote the dual projective space and likewise for any nonzero $\xi \in \R_3$ we denote by $[\xi]$ its corresponding point in $\mathbb{RP}_2$.

The group $\mathrm{SL}(3,\R)$ acts from the left on $\R^3$ by matrix multiplication and this action descends to define a transitive left action on $\mathbb{RP}^2$. Likewise, $\mathrm{SL}(3,\R)$ acts on $\R_3$ from the left by the rule
$$
h \cdot \xi=\xi h^{-1}
$$
for $h \in \mathrm{SL}(3,\R)$ and this actions descends to define a transitive left action on $\mathbb{RP}_2$. The stabiliser subgroup of $[x_0] \in \mathbb{RP}^2$ where $x_0={}^t(1\;0\;0)$ will be denoted by $H$, so that $\mathbb{RP}^2\simeq \mathrm{SL}(3,\R)/H$. The elements of $H\subset \mathrm{SL}(3,\R)$ are matrices of the form
$$
b\rtimes a=\left(\begin{array}{cc} \det a^{-1} & b \\ 0 & a\end{array}\right),
$$
with $a \in \mathrm{GL}(2,\R)$ and $b \in \R_2$. Denoting by $\mathbb{RP}_1\subset \mathbb{RP}_2$ the projective line consisting of those elements $[\xi] \in \mathbb{RP}_2$ which satisfy $[\xi]\cdot [x_0]=0$, the group $H$ acts faithfully from the left by affine transformations on the affine $2$-space $\mathbb{A}_2=\mathbb{RP}_2\setminus \mathbb{RP}_1$. Indeed, if we represent an element in $\mathbb{A}_2$ by a vector $(1,\xi) \in \R_3$ with $\xi \in \R_2$, we obtain
$$
(1,\xi)\left(\begin{array}{cc} \det a^{-1} & b \\ 0 & a\end{array}\right)^{-1}=\left(\det a,-ba^{-1}\det a+\xi a^{-1}\right)
$$
so that the induced affine transformation is
$$
(b\rtimes a)\cdot \xi=\xi a^{-1}\det a^{-1}-ba^{-1}.
$$    
Consequently, we may naturally think of $H$ as the $2$-dimensional real affine group.

\subsection{Projective structures}
\label{projective_sec}

In this preliminary subsection we shall summarise basic facts about projective 
structures on a surface which underlie the results of the paper; the reader may consult~\cite{BEG} for additional details. 
Let $\mf$ be a connected smooth surface. By an affine torsion-free connection on $N$ we mean a torsion-free connection on its tangent bundle $TN$. The set of torsion-free connections on $TN$ is an affine space modelled on the smooth sections of the vector bundle $V=TN\otimes S^2(T^*N)$. We have a canonical trace mapping $V \to T^*N$ and an inclusion $$
\iota : T^*N \to V, \quad \nu \mapsto \nu\otimes\mathrm{Id}+\mathrm{Id}\otimes \nu. 
$$
Consequently, $V$ decomposes into a direct sum $V\simeq V_0\oplus T^*N$, where $V_0$ denotes the trace-free part of $V$. 

The curvature $R^{\nabla}$ of the connection $\nabla$ is defined by 
$$
R^{\nabla}(X,Y)Z=\nabla_X\nabla_YZ-\nabla_Y\nabla_X Z-\nabla_{[X,Y]}Z
$$
for all vector fields $X,Y,Z$ on $N$. We define the Ricci curvature of $\nabla$ to be
$$
\mathrm{Ric}(\nabla)(X,Y)=\tr\left(Z\to R^{\nabla}(Z,X)Y\right)
$$
for all vector fields $X,Y$ on $N$.\footnote{This definition is common in projective differential geometry, but differs from the more standard definition, where the Ricci curvature is defined as $\mathrm{Ric}(\nabla)(X,Y)=\tr\left(Z\to R^{\nabla}(Z,Y)X\right)$.}  The Ricci curvature need not be symmetric and we denote by $\mathrm{Ric}^{\pm}(\nabla)$ its symmetric and anti-symmetric part.   

A vector field $K$ defined on some open set $U\subset N$ is called~\textit{affine} for the torsion-free connection $\nabla$ on $TN$ if its local flow $\phi_t$ preserves the geodesics of $\nabla$. The set of such vector fields on $U$ is a Lie subalgebra of the Lie algebra of vector fields on $U$ which we will denote by $\mathcal{A}_\nabla(U)$. Clearly, $K \in \mathcal{A}_{\nabla}(U)$ if and only if 
\begin{equation}\label{affvec}
0=\mathcal{L}_{K}\nabla:=\lim_{t\to 0} \frac{1}{t}\left(\phi_t^*\nabla-\nabla\right)
\end{equation}
on $U$. A straightforward computation yields that~\eqref{affvec} is equivalent to the vanishing of the symmetric part of $\nabla^2 K$. By definition, the map $K \mapsto \mathcal{L}_{K}\nabla$ takes values in $\Gamma(V)$ and hence defines a second order linear differential operator $\mathcal{L}^{\nabla} : \Gamma(TN) \to \Gamma(V)$.

A projective structure $[\nabla]$ on $\mf$ is an equivalence class
of torsion--free connections on $T\mf$, where two such connections $\hat{\nabla}$ and $\nabla$ are called~\textit{projectively equivalent} if they share the same unparametrised geodesics. By a classical result of Weyl~\cite{Weyl} this is equivalent to $\hat{\nabla}-\nabla$ being pure trace, that is, the existence of a $1$-form $\Upsilon$ on $N$ such that
\begin{equation}\label{connection_change}
\hat{\nabla}_XY=\nabla_XY+\Upsilon(X)Y+\Upsilon(Y)X,
\end{equation}
for all vector fields $X,Y$ on $N$. Consequently, the set of projective structures on $N$ is an affine space modelled on the smooth sections of $V_0$. 

Using index notation, the projective Schouten tensor $\Rho$ of $\nabla$ is defined by 
$$
\Rho_{ij}=R_{(ij)}+\frac{1}{3}R_{[ij]},
$$
where $R_{(ij)}$ denotes the symmetric part -- and $\mathrm{R}_{[ij]}$ the anti-symmetric part of the Ricci curvature of $\nabla$. 
If we change the connection in the
projective class using (\ref{connection_change}) then
\be
\label{schouten_change}
\hat{\Rho}_{ij}={\Rho}_{ij}-\nabla_i\Upsilon_j+\Upsilon_i\Upsilon_j, \quad
\hat{\Rho}_{[ij]}=\Rho_{[ij]}-\nabla_{[i}\Upsilon_{j]}.
\ee

A vector field $K$ defined on some open set $U\subset N$ is said to be~\textit{projective} for $[\nabla]$ if its local flow $\phi_t$ preserves the unparametrised geodesics of $[\nabla]$. The set of such vector fields on $U$ is a Lie subalgebra of the Lie algebra of vector fields on $U$ which we will denote by $\mathcal{P}_{[\nabla]}(U)$. A vector field $K$ belongs to $\mathcal{P}_{[\nabla]}(U)$ if and only if 
\begin{equation}\label{projsymvect}
0=\mathcal{L}_{K}[\nabla]:=\left(\mathcal{L}_{K}\nabla\right)_0
\end{equation}
on $U$, where $\nabla \in [\nabla]$, and the explicit expression
for $\mathcal{L}_{K}\nabla$ is given by (\ref{formula_for_lie}).
 By definition, the right hand side of~\eqref{projsymvect} is a smooth section of $V_0$ so that the map $K \mapsto \mathcal{L}_{K}[\nabla]$ defines a second order linear differential operator $\mathcal{L}^{[\nabla]} : \Gamma(TN) \to \Gamma(V_0)$.

If $N$ is orientable, we may restrict attention to connections in $[\nabla]$ which preserve an area form $\epsilon$ on $N$,  
so that $\nabla\epsilon=0$. We shall refer to such 
connections as {\em special} \cite{EM}. Note that special connections always exist globally. For special connections
the Schouten tensor is symmetric, that is $\Rho_{[ij]}=0$.
The residual freedom in special connections within 
a given projective class is given
by (\ref{connection_change}) where $\Upsilon=\d f$ for some smooth real-valued function $f$ on $N$. The special condition 
is preserved if  $\hat{\epsilon}=e^{3f}\epsilon$.

\subsection{The Cartan geometry of a projective surface}\label{cartgeomsec}

In~\cite{MR1504846} (see also~\cite{kobanag} for a modern reference), Cartan associates to a projective structure $[\nabla]$ on a smooth surface $\mf$ a Cartan geometry $(\pi : \projb \to \mf,\theta)$ of type $(\mathrm{SL}(3,\R),H)$ which consists of a principal right $H$-bundle $\pi : \projb \to \mf$ together with a Cartan connection $\theta\in \Omega^1(\projb,\mathfrak{sl}(3,\R))$ having the following properties:
\begin{itemize}
\item[(i)] $\theta(\bX_v)=v$ for fundamental vector field $\bX_v$ on $\projb$;
\item[(ii)] $\theta_u : T_u\projb\to \mathfrak{sl}(3,\R)$ is an isomorphism for all $u \in \projb$;
\item[(iii)] $R_h^*\theta=\mathrm{Ad}(h^{-1})\theta=h^{-1}\theta h$ for all $ h\in H$;  
 \item[(iv)] write
$$
\theta=\left(\begin{array}{cc} -\tr \phi & \eta \\ \omega & \phi\end{array}\right)
$$
for an $\R^2$-valued $1$-form $\omega=(\omega^i)$, an $\R_2$-valued $1$-form $\eta=(\eta_i)$ and a $\mathfrak{gl}(2,\R)$-valued $1$-form $\phi=(\phi^i_j)$. If $\bX_{x}$ is a vector field on $\projb$ having the property that
$$
\omega(\bX_{x})=x, \quad \eta(\bX_{x})=0, \quad \phi(\bX_{x})=0, 
$$
for some non-zero $x \in \R^2$, then the the integral curve of $\bX_{x}$, when projected to $\mf$, becomes a geodesic of $[\nabla]$ and conversely every geodesic of $[\nabla]$ arises in this way;  
\item[(v)] The curvature $2$-form $\Theta$ satisfies 
\begin{equation}\label{struceqproj}
\Theta=\d \theta+\theta\wedge\theta=\left(\begin{array}{cc} 0 & L(\omega\wedge\omega) \\ 0 & 0\end{array}\right),
\end{equation}
for a smooth curvature function $L : \projb \to \mathrm{Hom}\left(\R^2\wedge\R^2,\R_2\right)$.
\end{itemize} 
Note the Bianchi-identity
$$
\d \Theta=\Theta\wedge\theta-\theta\wedge\Theta,
$$
the algebraic part of which reads
\begin{equation}\label{algbianchi}
0=L(\omega\wedge\omega)\wedge\omega. 
\end{equation}
A projective structure $[\nabla]$ is called~\textit{flat} if locally $[\nabla]$ is defined by a flat connection. A consequence of Cartan's construction is that a projective structure is flat if and only if $L$ vanishes identically.

\begin{rmk}
Cartan's bundle is unique in the following sense: If $(\hat{\pi} : \hat{P}_{[\nabla]} \to \mf,\hat{\theta})$ is another Cartan geometry of type $(\mathrm{SL}(3,\R),H)$ satisfying the properties (iii),(iv),(v), then there exists a $H$-bundle isomorphism $\psi : \projb \to \hat{P}_{[\nabla]}$ so that $\psi^*\hat{\theta}=\theta$.
\end{rmk}

\begin{rmk}
Let $w$ be any real number. The line bundle associated to $\projb$ via the $H$-representation $\chi_w : H \to \mathrm{GL}^+(1,\R)$, $b\rtimes a \mapsto |\!\det a|^w$ will be denoted by $\mathcal{E}(w)$. Following~\cite{BEG}, we call its sections~\textit{densities of projective weight} $w$. In particular, nowhere vanishing sections of $\mathcal{E}(1)$ are known as~\textit{projective scales}.
\end{rmk}

\subsection{The choice of a representative connection}

For what follows it is necessary to have an explicit construction of the Cartan geometry $(\pi : \projb \to N,\theta)$ of a projective surface $(N,[\nabla])$. This can be achieved conveniently by fixing a representative connection $\nabla \in [\nabla]$. To this end let $\upsilon : F \to \mf$ denote the coframe bundle of $\mf$ whose fibre at at point $p \in \mf$ consists of the linear isomorphisms $u : T_p\mf \to \R^2$. The group $\mathrm{GL}(2,\R)$ acts transitively from the right on each $\upsilon$-fibre by the rule $R_a(u)=u \cdot a=a^{-1}\circ u$ for all $a \in \mathrm{GL}(2,\R)$. This action turns $\upsilon : F \to \mf$ into a principal right $\mathrm{GL}(2,\R)$-bundle. The bundle $F \to \mf$ is equipped with a tautological $\R^2$-valued $1$-form $\omega=(\omega^i)$ satisfying the equivariance property $(R_a)^*\omega=a^{-1}\omega$, where the $1$-form $\omega$ is defined by $\omega_u=u\circ \upsilon^{\prime}_{u}$.

Suppose $\varphi=(\varphi^i_j) \in \Omega^1(F,\mathfrak{gl}(2,\R))$ is the connection $1$-form of $\nabla \in [\nabla]$, then we have the structure equations
\begin{align}
\label{eq:struceqlincon}\d \omega^i&=-\varphi^i_j\wedge\omega^j,\\
\label{formula_for_rho}\d\varphi^k_l+\varphi^k_j\wedge\varphi^j_l&=\frac{1}{2}\left(\delta_i{}^k\Rho_{jl}-\delta_j{}^k\Rho_{il} -2\Rho_{[ij]}\delta_l{}^k\right)\omega^i\wedge\omega^j,
\end{align}
where -- by slight abuse of notation -- the $\R_2\otimes \R_2$-valued map $\Rho=(\Rho_{ij})$ represents the Schouten tensor of $\nabla$. We define a right $H$-action on $F\times \R_2$ by the rule
$$
(u,\xi)\cdot (b\rtimes a)=\left(\det a^{-1} a^{-1} \circ u,\xi a \det a -b\det a\right),
$$
for all $b\rtimes a \in H$ and $(u,\xi) \in F\times\R_2$. Denoting by $\pi : F\times \R_2 \to N$ the basepoint projection of the first factor, this action turns $\pi : F\times \R_2 \to \mf$ into a principal right $H$-bundle over $\mf$. On $F\times \R_2$ we define the $\mathfrak{sl}(3,\R)$-valued $1$-form
\begin{equation}\label{eq:cartanconrepcon}
\theta=\left(\begin{array}{cc} -\frac{1}{3}\tr \varphi+\xi \omega & -\d \xi+\xi\varphi-\omega^t\Rho^t-\xi\omega\xi\\  \omega & \varphi-\frac{1}{3}\mathrm{I}\tr \varphi-\omega\xi  \end{array}\right).
\end{equation}
Then $(\pi : F \times \R_2 \to \mf,\theta)$ is a Cartan geometry of type $(\mathrm{SL}(3,\R),H)$ satisfying the properties (iii) to (v) for the projective structure defined by $\nabla$. It follows from the uniqueness part of Cartan's construction that $(\pi : F \times \R_2\to\mf,\theta)$ is isomorphic to the Cartan geometry of $(N,[\nabla])$.

\subsection{The Patterson-Walker metric}

In~\cite{patwalk}, Patterson and Walker use an affine torison-free connection $\nabla$ on a smooth manifold to construct a split-signature metric on its cotangent bundle. Here we briefly review their construction for the case of a surface $N$. As before, let $\upsilon : F \to N$ denote the coframe bundle of $N$ with tautological $1$-form $\omega$ and let $\varphi$ denote the connection form of $\nabla$. The cotangent bundle $\nu : T^*N \to N$ is the bundle associated to the $\mathrm{GL}(2,\R)$-representation $\chi$ on $\R_2$ defined by the rule
$
\chi(a)\xi=\xi a^{-1}
$ 
for all $a \in \mathrm{GL}(2,\R)$ and $\xi \in \R_2$. The $1$-forms on $F\times \R_2$ that are semi-basic for the projection $\zeta : F\times \R_2 \to T^*N\simeq (F\times \R_2)/\sim_{\chi}$ are spanned by the components of $\omega$ and $\d \xi-\xi\varphi$. In particular, the equivariance properties of $\omega,\theta$ and $\xi$ imply that the tensor field $\left(\d \xi-\xi\varphi\right)\omega=\left(\d\xi_i-\xi_k\varphi^k_i\right)\otimes\omega^i$ is invariant under the $\mathrm{GL}(2,\R)$-right action,
$$
(R_a)^*\left(\d \xi-\xi\varphi\right)\omega=\left( \d \xi a-\xi a a^{-1}\varphi a\right)a^{-1}\omega=\left(\d \xi-\xi\varphi\right)\omega. 
$$
It follows that there exists a unique split-signature metric $h_{\nabla}$ and a unique $2$-form $-\Omega_0$ on $T^*N$ such that
$$
\zeta^*h_{\nabla}=\left(\d\xi_i -\xi_k\varphi^k_i\right)\circ \omega^i\quad \text{and}\quad \zeta^*\Omega_0=-\left(\d\xi_i -\xi_k\varphi^k_i\right)\wedge \omega^i.
$$
Note that the $1$-form $\xi\omega$ is semi-basic for the projection $\zeta$ and invariant under the $\mathrm{GL}(2,\R)$-right action, hence the pullback of a unique $1$-form $\lambda$ on $T^*N$ which is of course the tautological $1$-form (or Liouville $1$-form) of $T^*N$. The structure equation~\eqref{eq:struceqlincon} gives
$$
-\d (\xi_i\omega^i)=-\d\xi_i\wedge\omega^i+\xi_k\varphi^k_i\wedge\omega^i,
$$
hence $\Omega_0=-\d \lambda$ is just the canonical symplectic form of $T^*N$ and independent of $\nabla$. The metric $h_{\nabla}$ does however depend on $\nabla$ and is called the~\textit{Patterson-Walker metric} or the~\textit{Riemannian extension} of $\nabla$. In canonical local coordinates $(x^i,\xi_i)$ on an open subset of the cotangent bundle it takes the form
\begin{equation}\label{walker}
h_{\nabla}=d\xi_i\odot dx^i-\Gamma_{ij}^k\;\xi_k\; dx^i\odot dx^j,
\end{equation}
where $\Gamma^i_{jk}$ denote the Christoffel symbols of $\nabla$ with respect to the coordinates $(x^i)$.

\subsection{Anti-self-duality}
Let $M$ be an oriented four--dimensional manifold with a metric
$g$ of signature $(2, 2)$. The Hodge $\ast$ operator is an
involution on two-forms, and induces a decomposition
\be \label{splitting}
\Lambda^{2}(T^*M) = \Lambda_{+}^{2}(T^*M) \oplus \Lambda_{-}^{2}(T^*M)
\ee
of two-forms
into self-dual (SD)
and anti-self-dual (ASD)  components, which
only depends on the conformal class of $g$. 
The Riemann tensor of $g$ 
has the symmetry $R_{abcd}=R_{[ab][cd]}$ so can be thought of
as a map $\mathcal{R}: \Lambda^{2}(T^*M) \rightarrow \Lambda^{2}(T^*M)$
which admits a decomposition   under (\ref{splitting}):
\be \label{decomp}
{\mathcal R}=
\left(
\mbox{
\begin{tabular}{c|c}
&\\
$C_+-2\Lambda$&$\phi$\\ &\\
\cline{1-2}&\\
$\phi$ & $C_--2\Lambda$\\&\\
\end{tabular}
} \right) .
\ee
Here $C_{\pm}$ are the SD and ASD parts 
of the (conformal) Weyl tensor, $\phi$ is  the
trace-free Ricci curvature, and $-24\Lambda$ is the scalar curvature which acts
by scalar multiplication. 
The metric $g$ is ASD if $C_{+}=0$. It is ASD and Einstein if 
$C_{+}=0$ and $\phi=0$. Finally it is ASD Ricci--flat 
(or equivalently hyper-symplectic) if $C_{+}=\phi=\Lambda=0$. In this case
the Riemann tensor is also anti-self-dual.
\vskip3pt
Locally there exist real rank-two vector bundles $\spp, \spp'$  (spin-bundles) over $M$ equipped with parallel symplectic structures
$\ep, \ep'$ such that
\be
\label{can_bun_iso}
T M\cong {\spp}\otimes {\spp'}
\ee
is a  canonical bundle isomorphism, and
\[
g(v_1\otimes w_1,v_2\otimes w_2)
=\varepsilon(v_1,v_2)\varepsilon'(w_1, w_2)
\]
for $v_1, v_2\in \Gamma(\spp)$ and $w_1, w_2\in \Gamma(\spp')$.
A vector $V\in \Gamma(TM)$ is called null if $g(V, V)=0$. Any null vector is of the form
$V=\lambda \otimes \pi$ where $\lambda$, and $\pi$ are sections of
$\spp$ and $\spp'$ respectively.
An $\alpha$--plane (respectively a $\beta$--plane) 
is a two--dimensional plane in $T_pM$
spanned by null vectors of the above form with $\pi$ (respectively
$\lambda$) fixed, and
an $\alpha$--surface ($\beta$--surface) is a two--dimensional surface in $\zeta\subset M$ such 
that its tangent plane at every point is an $\alpha$--plane ($\beta$--plane). The 
seminal theorem  of Penrose \cite{penrose} states that   a 
maximal, three dimensional, family of $\alpha$--surfaces exists 
in $M$ iff $C_+=0$.

\section{From projective to bi-Lagrangian structures}

In this section we show how to canonically construct a bi-La\-grang\-ian structure on the total space of a certain rank $2$ affine bundle over a projective surface $(N,[\nabla])$. Recall that the group $H$ also acts faithfully on $\R_2$ by affine transformations defined by the rule
\begin{equation}\label{eq:affineaction}
\left(b\rtimes a\right)\cdot \xi=\xi a^{-1}\det a^{-1}-ba^{-1}
\end{equation}
for all $\xi \in \R_2$ and $b\rtimes a \in H$. Therefore, the bundle associated to $\projb$ via this affine $H$-action is a rank-$2$ affine bundle $M\to \mf$. We will refer to $M$ as the~\textit{canonical affine bundle} of $(\mf,[\nabla])$.

By definition, an element of $M$ is an equivalence class $[u,\xi]$ with $u\in \projb$ and $\xi \in \R_2$ subject to the equivalence relation
$$
(u_1,\xi_1)\sim (u_2,\xi_2) \quad \iff \quad u_2=u_1\cdot b\rtimes a\;\;\land \;\; \xi_2=(b\rtimes a)^{-1}\cdot \xi_1, \quad b\rtimes a \in H.  
$$  
Clearly, every element of $M$ has a representative $(u,0)$, unique up to a $\mathrm{GL}(2,\R)$ transformation, where here $\mathrm{GL}(2,\R)\subset H$ consists of those elements $b\rtimes a \in H$ satisfying $b=0$. For simplicity of notation, we will henceforth write $a$ instead of $0\rtimes a$ for the elements of $\mathrm{GL}(2,\R)\subset H$. It follows that as a smooth manifold, $M$ is canonically diffeomorphic to the quotient $\projb/\mathrm{GL}(2,\R)$ and we let $\mu : \projb \to M$ denote the quotient projection.
\begin{rmk}
It can be shown that the sections of $M \to N$ are in one-to-one correspondence with the $[\nabla]$-representative connections. The submanifold geometry in $M$ of representative connections is studied in depth in two articles by the second author~\cite{tm15,tm16}.
\end{rmk}
We use the standard fact that the tangent bundle of $N$ is the bundle associated to $\projb$ via the natural $H$-action on $\mathfrak{sl}(3,\R)/\mathfrak{h}$ induced by the adjoint representation of $H$ on its Lie algebra $\mathfrak{h}$. An element in the Lie algebra $\mathfrak{sl}(3,\R)$ of $\mathrm{SL}(3,\R)$ can be written as
$$
m_{x,\xi,\alpha}=
\begin{pmatrix}
-\tr \alpha & \xi \\ x & \alpha
\end{pmatrix} ,
$$
where $x \in \R^2,\xi\in\R_2, \alpha\in \mathfrak{gl}(2,\R)$ and $\mathfrak{h}$ consists of those elements for which $x=0$. Therefore, the elements in the quotient $\mathfrak{sl}(3,\R)/\mathfrak{h}\simeq \R^2$ are uniquely represented by matrices of the form $m_{x,0,0}$. Hence the induced action of $H$ is
\begin{equation}\label{eq:inducedaction}
\left(b\rtimes a\right)
\begin{pmatrix}
0 & 0 \\ x & 0 
\end{pmatrix}\left(b\rtimes a\right)^{-1}=\begin{pmatrix} 0 & 0 \\(\det a)a x & 0\end{pmatrix}\quad \text{mod} \;\mathfrak{h}. 
\end{equation}
In particular, since the cotangent bundle of $\mf$ is the bundle associated to the representation $\chi : H \to \mathrm{GL}(\R_2)$ which is dual to the representation defined by~\eqref{eq:inducedaction}, it follows that $\chi$ is defined by the rule
$$
\chi(b\rtimes a)\xi=\xi a^{-1}\det a^{-1},
$$
for all $\xi \in \R_2$ and $b\rtimes a \in H$.

Since $\chi$ is precisely the linear part of the affine $H$-action~\eqref{eq:affineaction}, we see that the affine bundle $M \to N$ is modelled on the cotangent bundle of $N$. 

\subsection{A bundle embedding}\label{subsec:bundleembedding}

It turns out that we can embed $\projb \to M$ as subbundle of the coframe bundle $F\to M$ of $M$. Here, we define a coframe at $p \in M$ to be a linear isomorphism $T_pM \to \R_2\oplus \R^2$ and we denote the tautological $\R_2\oplus\R^2$-valued $1$-form on $F$ by $\zeta$. 

By definition of $M$, a vector field $\bX$ on $M$ is represented by a unique $(\R_2\oplus \R^2)$-valued function $(X_+,X_-)$ on $\projb$ satisfying the equivariance condition
\begin{equation}\label{anothereq}
R_{a}^*X_+=X_+ a\det a, \quad R_{a}^*X_-=(\det a^{-1})a^{-1}X_-. 
\end{equation}
Therefore, we obtain a unique map $\psi : \projb \to F$ having the property that for every vector field $\bX$ on $M$ and for all $u \in \projb$  
$$
\psi(u)(\bX(\mu(u)))=(X_+(u),X_-(u)),
$$
where $(X_+,X_-)$ is the function on $\projb$ representing $\bX$. Clearly, $\psi$ is a smooth embedding.  Furthermore, from~\eqref{anothereq} we obtain
$$
\psi(u\cdot a)=\psi(u)\cdot\chi(a)
$$
where $\chi : H\ni \mathrm{GL}(2,\R) \to \mathrm{Aut}(\R_2\oplus\R^2)$ is the Lie group embedding defined by the rule
$$
\chi\left(a\right)(\xi,x)=\left(\xi a \det a,(\det a^{-1})a^{-1} x\right).
$$
Consequently, the pair $(\psi,\chi)$ embeds $\projb \to M$ as a subbundle of the coframe bundle of $M$ whose structure group is isomorphic to $\mathrm{GL}(2,\R)$. Furthermore, unraveling the definition of $\zeta$, it follows that we have
\begin{equation}\label{pullbackcanonical}
\psi^*\zeta=(\eta,\omega). 
\end{equation}
The induced geometric structure on $M$ defined by the reduction of the coframe bundle of $M$ is a bi-La\-grang\-ian structure, so we will study these structures next. 
 
\subsection{Bi-La\-grang\-ian structures}

A \textit{bi-La\-grang\-ian} structure on smooth $4$-mani\-fold $M$ (or more generally an even dimensional manifold) consists of a symplectic structure $\Omega$ together with a splitting of the tangent bundle of $M$ into a direct sum of $\Omega$-La\-grang\-ian subbundles $E_{\pm}$
$$
TM= E_+\oplus E_-. 
$$
A manifold equipped with a bi-La\-grang\-ian structure will be called a bi-La\-grang\-ian manifold. The endomorphism $I : TM \to TM$ defined by
$$
\bv=\bv_++\bv_-\mapsto \bv_+-\bv_-, \quad \bv \in TM, \bv_{\pm} \in E_{\pm}
$$
is $\Omega$-skew-symmetric, satisfies $I^2=\mathrm{Id}$ and its $\pm 1$-eigenbundle is $E_{\pm}$. Clearly, $I$ is the unique endomorphism of the tangent bundle having these properties and therefore, we may equivalently think of a bi-La\-grang\-ian structure as a pair $(\Omega,I)$ consisting of a symplectic structure $\Omega$ and a $\Omega$-skew-symmetric endomorphism $I : TM \to TM$ whose square is the identity.

Note also, that we may use the pair $(\Omega,I)$ to define a pseudo-Riemannian metric 
$$
g(\bv,\bw)=\Omega(\bv,I(\bw)), \quad \bv,\bw \in TM, 
$$
whose signature is $(2,2)$ and for which $I$ is skew-symmetric. Of course, a bi-La\-grang\-ian structure is also equivalently described in terms of the pair $(g,I)$ or the pair $(g,\Omega)$. 

\begin{rmk}
What we call a bi-La\-grang\-ian structure is also referred to as an~\textit{almost para-K\"ahler structure} and a~\textit{para-K\"ahler structure} provided $E_{\pm}$ are both Frobenius integrable. Note that in~\cite{bryantboch} the term bi-La\-grang\-ian structure is reserved for the case where both $E_{\pm}$ are Frobenius integrable. 
\end{rmk} 

\begin{rmk}
We call a vector field defined on some open subset $U\subset (M,\Omega,I)$~\textit{bi-La\-grang\-ian} if its (local) flow preserves both $\Omega$ and $I$. The set of such vector fields on $V$ is a Lie subalgebra of the Lie algebra of vector fields on $V$ which we will denote by $\mathcal{B}_{(\Omega,I)}(U)$.
\end{rmk}

A bi-La\-grang\-ian structure admits an interpretation as a reduction of the structure group of the coframe bundle of $M$. To this end consider the symmetric bilinear form of signature $(2,2)$ on $\R_2\oplus \R^2$
$$
\left\langle (\xi_1,x^1),(\xi_2,x^2)\right\rangle=-\frac{1}{2}\left(\xi_1x^2+\xi_2x^1\right)
$$
and the skew-symmetric non-degenerate bilinear form
$$
\rangle(\xi_1,x^1),(\xi_2,x^2)\langle=\frac{1}{2}\left(\xi_1x^2-\xi_2x^1\right).
$$
The two bilinear forms are related by the endomorphism $\iota$ sending $(\xi,x) \mapsto (\xi,-x)$. The endomorphism $\iota$ satisfies $\iota^2=\mathrm{Id}$ and its $1$-eigenspace is $\R_2\oplus \{0\}$ and its $-1$-eigenspace is $\{0\}\oplus \R^2$. By construction, both eigenspaces are null and La\-grang\-ian, that is, both bilinear forms vanish identically when restricted to the $\iota$-eigenspaces. The group $\mathrm{GL}(2,\R)$ acts from the left on $\R_2\oplus \R^2$ by
$$
a\cdot (\xi,x)=\left(\xi a^{-1},ax\right)
$$
and this action preserves both bilinear forms. We henceforth identify $\mathrm{GL}(2,\R)$ with its image subgroup in $\mathrm{Aut}( \R_2\oplus \R^2)$. In fact, $\mathrm{GL}(2,\R)$ is the largest subgroup of $\mathrm{Aut}(\R_2\oplus \R^2)$ preserving both bilinear forms.   

Given a bi-La\-grang\-ian structure $(\Omega,I)$ on $M$ we say that a coframe $u$ at $p \in M$ is~\textit{adapted} to $(\Omega,I)$ if for all $\bv, \bw\in T_pM$ 
$$
\Omega_p(\bv,\bw)=\rangle u(\bv),u(\bw)\langle \quad \text{and}\quad \left(u\circ I\right)(\bv)=\left(\iota\circ u\right)(\bv). 
$$
The set of all coframes of $M$ adapted to $(\Omega,I)$ defines a reduction $\lambda : \bilb \to M$ of the coframe bundle $F \to M$ of $M$ with structure group $\mathrm{GL}(2,\R)$. Conversely, every reduction of the coframe bundle of $M$ with structure group $\mathrm{GL}(2,\R)$ defines a unique pair $(\Omega,I)$, consisting of a non-degenerate $2$-form on $M$ and a $\Omega$-skew symmetric endomorphism $I : TM \to TM$ whose square is the identity. Note however that $\Omega$ need not be closed.   

The tautological $\R_2\oplus\R^2$-valued $1$-form $\zeta$ on $\bilb$ will be written as $\zeta=(\eta,\omega)$, so that $\eta=(\eta_i)$ is an $\R_2$-valued $1$-form on $\bilb$ and $\omega=(\omega^i)$ is an $\R^2$-valued $1$-form on $\bilb$. By construction, we have
$$
\lambda^*\Omega=-\eta\wedge\omega:=-\eta_i\wedge\omega^i.
$$
Furthermore, let $\hat{L}_{\pm}=\left(\lambda^{\prime}\right)^{-1}(E_{\pm})\subset T\bilb$, then the subbundle $\hat{L}_+$ is defined by the equations $\eta=0$ and the subbundle $\hat{L}_-$ is defined by the equations $\omega=0$.

A linear connection on $F$ is said to be adapted to $(\Omega,I)$ if it pulls back to $\bilb$ to become a principal $\mathrm{GL}(2,\R)$-connection on $\bilb$. An adapted connection is given by a $\mathfrak{gl}(2,\R)$-valued equivariant $1$-form $\nu$ on $\bilb$ such that
$$
\aligned
\d\eta&=-\eta\wedge\nu+\frac{1}{2}T_+\left((\eta,\omega)\wedge(\eta,\omega)\right),\\
\d \omega&=-\nu\wedge\omega+\frac{1}{2}T_-\left((\eta,\omega)\wedge(\eta,\omega)\right),\\
\endaligned
$$
for some torsion map $T_+$ on $\bilb$ with values in $\mathrm{Hom}(\Lambda^2(\R_2\oplus \R^2),\R_2)$ and some torsion map $T_-$ on $\bilb$ with values in $\mathrm{Hom}(\Lambda^2(\R_2\oplus \R^2),\R^2)$, both of which are equivariant with respect to the $\mathrm{GL}(2,\R)$ right action. It is an easy exercise in linear algebra to check that for every bi-La\-grang\-ian structure there exists a unique adapted connection $\nu$ so that
\begin{equation}\label{firststruceqbil}
\aligned
\d\eta&=-\eta\wedge\nu+\frac{1}{2}T_+\left(\omega\wedge\omega\right),\\
\d \omega&=-\nu\wedge\omega+\frac{1}{2}T_-\left(\eta\wedge\eta\right),\\
\endaligned
\end{equation}
with $T_+$ taking values in $\mathrm{Hom}(\Lambda^2\R^2,\R_2)$ and $T_-$ taking values in $\mathrm{Hom}(\Lambda^2\R_2,\R^2)$. It follows that $E_{\pm}$ is integrable if and only if $T_{\pm}$ vanishes identically. Furthermore, the identity $\d(\eta\wedge\omega)=0$ implies
$$
T_+(\omega\wedge\omega)\wedge\omega=0\quad \text{and}\quad T_-(\eta\wedge\eta)\wedge\eta=0.
$$
 The linear connection $\nu$ on the bundle of adapted frames induces connections on the tensor bundles of $M$ in the usual way. By construction, the induced connection ${}^{\nu}\nabla$ on $TM$ is the unique (affine) connection  with torsion $\tau$ satisfying
$$
{}^{\nu}\nabla \Omega=0 \quad \text{and}\quad {}^{\nu}\nabla I=0 \quad\text{and}\quad \tau(X_+,X_-)=0, 
$$
for all $X_{\pm} \in \Gamma(E_{\pm})$. To the to best of our knowledge, the connection ${}^{\nu}\nabla$ was first studied by Libermann~\cite{libermann}, so we call $\nu$ the~\textit{Libermann connection}. Of course, if $\tau$ vanishes identically, then ${}^{\nu}\nabla$ is just the Levi-Civita connection of $g$. 

\subsection{From projective to bi-La\-grang\-ian structures}

Denoting by $\bilb$ the bundle of adapted coframes of a bi-La\-grang\-ian structure $(\Omega,I)$ and by $\projb$ the Cartan bundle of a projective structure $[\nabla]$, we obtain:
\begin{theo}\label{maincorrespondence}
Let $(\mf,[\nabla])$ be a projective surface with Cartan bundle $(\pi : \projb \to \mf,\theta)$. Then there exists a bi-La\-grang\-ian structure $(\Omega,I)$ on the quotient $M=\projb/\mathrm{GL}(2,\R)$ having the following property: There exists a $\mathrm{GL}(2,\R)$-bundle isomorphism $\psi : \projb \to \bilb$ so that
$$
\psi^*\left(\begin{array}{cc} -\frac{1}{3}\tr \nu & \eta \\ \omega & \nu-\frac{1}{3}\mathrm{Id}\tr \nu\end{array}\right)=\theta, 
$$
where $(\eta,\omega)$ denotes the tautological $1$-form on $\bilb$ and $\nu$ the Libermann connection. Moreover, the $E_-$-bundle of the bi-Lagrangian structure $(\Omega,I)$ is always Frobenius integrable and the $E_+$-bundle is Frobenius integrable if and only if $[\nabla]$ is flat.  
\end{theo}
\begin{proof}
We write
$$
\theta=\left(\begin{array}{cc} -\tr \phi & \hat{\eta} \\ \hat{\omega} & \phi\end{array}\right)
$$
for the Cartan connection on $\projb$. From \S\ref{subsec:bundleembedding} we know that we have an embedding $(\psi,\chi)$ of $\projb \to M$ as a $\mathrm{GL}(2,\R)$-subbundle $\lambda : \bilb \to M$ of the coframe bundle of $M$. Furthermore, $\psi$ satisfies 
$$
\psi^*(\eta,\omega)=(\hat{\eta},\hat{\omega}),
$$ 
where $(\eta,\omega)$ denotes the tautological $(\R_2\oplus\R^2)$-valued $1$-form on $\bilb$. Therefore, we obtain a unique non-degenerate $2$-form $\Omega$ on $M$ and a unique $\Omega$-skew-symmetric endomorphism $I : TM \to TM$ whose square is the identity. The $2$-form $\Omega$ pulled back to $\bilb$ becomes $-\eta\wedge \omega$. The structure equations~\eqref{struceqproj} imply that we have
\begin{equation}\label{struceqpart}
\aligned
\d \hat{\omega}&=-(\phi+\mathrm{I}\tr \phi)\wedge\hat{\omega},\\
\d \hat{\eta}&=-\hat{\eta}\wedge(\phi+\mathrm{I}\tr \phi)+L(\hat{\omega}\wedge\hat{\omega}). 
\endaligned
\end{equation}
In particular, we obtain
$$
\aligned
\d \left(\hat{\eta}\wedge\hat{\omega}\right)&=\left[-\hat{\eta}\wedge(\phi+\mathrm{I}\tr \phi)+L(\hat{\omega}\wedge\hat{\omega})\right]\wedge\hat{\omega}-\hat{\eta}\wedge\left[-(\phi+\mathrm{I}\tr \phi)\wedge\hat{\omega}\right]\\
&=L\left(\hat{\omega}\wedge\hat{\omega}\right)\wedge\hat{\omega}=0,
\endaligned
$$
where the last equality follows since $\mf$ is two-dimensional. This shows that $\Omega$ is symplectic, so that the pair $(\Omega,I)$ defines a bi-La\-grang\-ian structure on $M$. The equivariance properties of $\theta$ and~\eqref{struceqpart} imply that the $\psi$-pushforward of $\phi+\mathrm{I}\tr\phi$ is a principal right $\mathrm{GL}(2,\R)$-connection on $\bilb$ which satisfies~\eqref{firststruceqbil} with $T_-\equiv 0$ and $T_+=L\circ \psi^{-1}$. In particular, $E_-$ is always integrable and $E_+$ is integrable if and only if $L$ vanishes identically, that is, $[\nabla]$ is flat. Denoting by $\nu$ the Libermann connection of $(\Omega,I)$, we obtain from its uniqueness that
\begin{equation}\label{pullblibermann}
\psi^*\nu=\phi+\mathrm{I}\tr \phi, 
\end{equation}
which completes the proof. 
\end{proof}

\begin{rmk}
Recall that if $\bX_{x}$ is a vector field on $\projb$ having the property that
$$
\omega(\bX_{x})=x, \quad \eta(\bX_{x})=0, \quad \phi(\bX_{x})=0, 
$$
for some non-zero $x \in \R^2$, then the the integral curve of $\bX_{x}$, when projected to $\mf$, becomes a geodesic of $[\nabla]$. Conversely every geodesic of $[\nabla]$ arises in this way. Likewise, a geodesic of the Libermann connection arises as the projection of an integral curve of a horizontal vector field on $\bilb$ which is constant on the canonical $1$-form. It follows that the geodesics on $(\mf,[\nabla])$ correspond to the geodesics of the Libermann connection on $(M,\Omega,I)$ that are everywhere tangent to $E_-$.  
\end{rmk}

\subsection{A local coordinate descripition}

Recall from \S\ref{cartgeomsec} that the choice of a representative connection $\nabla \in [\nabla]$ gives a $H$-bundle isomorphism $\projb\simeq F \times \R_2$. In particular, we obtain a diffeomorphism $\psi_{\nabla} : (F\times\R_2)/\mathrm{GL}(2,\R) \to M$. By construction, the quotient $(F\times\R_2)/\mathrm{GL}(2,\R)$ is the cotangent bundle of $\mf$. Denoting the projection $ F\times \R_2 \to T^*\mf$ by $\mu$ as well, we obtain 
\begin{equation}\label{metrisymppreferred}
\aligned
(\psi\circ \mu)^*g&=-\left(-\d \xi+\xi\varphi-\Rho^t\omega-\xi\omega\xi\right)\odot \omega,\\
(\psi\circ \mu)^*\Omega&=\omega\wedge\left(-\d \xi+\xi\varphi-\Rho^t\omega-\xi\omega\xi\right),
\endaligned
\end{equation}
where the $\R_2\otimes \R_2$-valued map $\Rho=(\Rho_{ij})$ on $F$ represents the Schouten tensor of $\nabla$ and $\varphi$ the connection form of $\nabla$. Using~\eqref{metrisymppreferred}, we see that in terms of the Patterson-Walker metric $h_{\nabla}$ of $\nabla$ and the Liouville $1$-form $\lambda$ of $T^*N$, the metric can be expressed as
\begin{equation}\label{eq:metricchoicecon}
g_{\nabla}:=(\psi_{\nabla})^*g=h_{\nabla}+\lambda^2+\nu^*\mathrm{Ric}^+(\nabla)
\end{equation}
and for the symplectic form we obtain
$$
\Omega_{\nabla}:=(\psi_{\nabla})^*\Omega=-\Omega_0+\frac{1}{3}\nu^*\mathrm{Ric}^{-}(\nabla).
$$
In canonical local coordinates $(x^i,\xi_i)$ on $T^*N$, we thus have the expressions
\begin{equation}\label{coordexpmetsymp}
\aligned
g_{\nabla}&=\left(\d \xi_i\odot\d x^i -\left(\xi_l\Gamma^l_{ij}-\Rho_{(ij)}-\xi_i\xi_j\right)\d x^i\odot \d x^{j}\right),\\
\Omega_{\nabla}&=\d \xi_i \wedge \d x^i+\Rho_{[ij]}\d x^i\wedge \d x^j,
\endaligned
\end{equation}
where $\Gamma^i_{jk}$ denote the Christoffel symbols and $\Rho_{ij}$ the components of the Schouten tensor of $\nabla$ with respect to the coordinates $x^i$.

\begin{rmk}

Besides taking the quotient of the Cartan bundle by $\mathrm{GL}(2,\R)$, one might also consider the quotient by $\R^2\rtimes H$, where $H$ is the connected nonabelian real Lie group of dimension two. This quotient -- which is a formal analogue to the construction of the conformal Fefferman metrics~\cite{Fef} -- was studied in~\cite{NurSpar}. We also refer the reader to~\cite{hamerl2} for a generalisation of this construction to higher dimensions and its relation to the classical Patterson--Walker metrics~\cite{patwalk}.  

\end{rmk}

\subsection{Lift of projective vector fields}

Denoting by $\rho : M \to \mf$ the basepoint projection, an immediate consequence of Theorem~\ref{maincorrespondence} is:
\begin{col}\label{vectisom}
For every open set $U\subset \mf$ the Lie algebra of projective vector fields $\mathcal{P}_{[\nabla]}(U)$ is isomorphic to the Lie algebra of bi-La\-grang\-ian vector fields $\mathcal{B}_{(\Omega,I)}(\rho^{-1}(U))$. 
\end{col}
\begin{proof}
By standard results about Cartan geometries (c.f.~\cite{parabook}), the projective vector fields on $U\subset (N,[\nabla])$ are in one-to-one correspondence with the vector fields on $\pi^{-1}(U)\subset \projb$ whose flow preserves the Cartan connection $\theta$ and which are equivariant for the principal right action. Theorem~\ref{maincorrespondence} implies that such a vector field corresponds to a vector field on $\psi(\pi^{-1}(U))\subset \bilb$ preserving both the tautological form $(\eta,\omega)$ and the Libermann connection. Again, by standard results about $G$-structures~\cite{parabook}, such vector fields are in one-to-one correspondence with vector fields on $\rho^{-1}(U)$ preserving both $\Omega$ and $I$.   
\end{proof}

Corollary~\ref{vectisom} can be strengthened in the sense that we show that every Killing vector field for $(M,g)$ is also symplectic with respect to $\Omega$ and hence the lift of a projective vector field on $(N,[\nabla])$. As a warm up, we first consider a correspondence between affine vector fields and Killing vector fields for the asscoiated Patterson--Walker metric 
(\ref{walker}). Let $\nabla$ be an affine connection on $\mf$. Recall that a vector field $K$ 
on $\mf$ is affine with respect to $\nabla$
if and only if 
\be
\label{formula_for_lie}
0=({\mathcal L}_K \nabla)_{ij}^k\equiv \frac{\p^2 K^k}{\partial x^i \partial x^j}
+K^m\frac{\p}{\partial x^m} \Gamma_{ij}^k- \Gamma_{ij}^m\frac{\partial K^k}{\partial x^m}
+\Gamma_{im}^k\frac{\partial K^m}{\partial x^j}+\Gamma_{jm}^k\frac{\partial K^m}{\partial x^i},
\ee
where we write $K=K^i\frac{\partial}{\partial x^i}$ in local coordinates $(x^i)$ on $U\subset N$ and where $\Gamma^i_{jk}$ denote the Christoffel symbols of $\nabla$ with respect to $(x^i)$. 
Any vector field on $\mf$ corresponds to a linear function on $T^*\mf$,
which in canonical local coordinates $(x^i,\xi_i)$ is given by $K^i\xi_i$. This function, together
with the canonical symplectic structure on $T^*\mf$ gives rise to the 
Hamiltonian vector field
\be
\label{complete_lift}
\widetilde{K}=K^i\frac{\p}{\partial x^i}-\xi_j\frac{\partial K^j}{\partial x^i}\frac{\p}{\partial \xi_i}.
\ee
This vector field is sometimes referred to as the complete lift \cite{yano}.
\begin{prop}
\label{prop_kill_1}
Let $K$ be an affine vector field for a connection $\nabla$ on $U\subset N$.
 Then
its complete lift  (\ref{complete_lift}) is a Killing vector field
for the Patterson-Walker metric (\ref{walker}).
\end{prop}
\noindent 
\begin{proof}
Consider the one--parameter group of transformations generated by the vector 
field (\ref{complete_lift})
\[
x^i\longrightarrow x^i+\epsilon\; K^i+O(\epsilon^2), \quad \xi_i\longrightarrow
\xi_i-\epsilon\; \xi_j\frac{\partial K^j}{\partial x^i}+O(\epsilon^2).
\]
This yields
\begin{eqnarray*}
g&&\longrightarrow g+\epsilon\{\xi_j K^idx^j d\xi_i-\xi_i K^j dx^id\xi_j
-(\xi_j\xi_i\xi_k K^j)dx^idx^k\\
&&-2\Gamma_{ik}^j\xi_j(\xi_m K^i)dx^kdx^m
+\Gamma_{ik}^j\xi_m\xi_j K^m dx^idx^k-K^m(\xi_m \Gamma_{ik}^j)\xi_j dx^idx^k\}
+O(\epsilon^2)
\\
&&=g-\epsilon\; \xi_k{\mathcal L}_{{K}} (\Gamma_{ij}^k) dx^i\odot dx^j+O(\epsilon^2).
\end{eqnarray*}
Therefore
\be
\label{formula}
{\mathcal L}_{\widetilde{K}} g=-
\xi_k{\mathcal L}_{{K}} (\Gamma_{ij}^k) dx^i\odot dx^j,
\ee
and the result follows.
\end{proof}
Recall that a vector field $K$ is projective for $\nabla$ if and only if $(\mathcal{L}_K\nabla)_0=0$, that is, there exists a $1$-form $\rho$ on $N$ such that
\be
\label{projective_field}
({\mathcal L}_K \nabla)^k_{ij}={\delta_i}^k \rho_j+ {\delta_j}^k \rho_i.
\ee
\begin{prop}
Let $K$ be a projective vector field with $\rho_i=\nabla_i f$. Then
\be
\label{vect_theo}
K-\xi_j\frac{\partial K^j}{\partial x^i}\frac{\p}{\partial \xi_i}+f\xi_i\frac{\p}{\partial \xi_i}
\ee
is a conformal Killing vector field for the Patterson-Walker metric
(\ref{walker}).
\end{prop}
\noindent
\begin{proof}
The proof is similar to that of Proposition (\ref{prop_kill_1}). The one-parameter
group of transformation generated by (\ref{vect_theo}) is
\[
x^i\longrightarrow x^i+\epsilon\; K^i+O(\epsilon^2), \quad \xi_i\longrightarrow
\xi_i-\epsilon\; \xi_j\frac{\partial K^j}{\partial x^i} -\epsilon f\xi_i +O(\epsilon^2),
\]
which gives
\[
g\longrightarrow 
g-\epsilon\; \xi_k\left({\mathcal L}_{{K}}\nabla\right)_{ij}^k dx^i\odot dx^j
-\xi_kdx^k\odot df+\epsilon f\;g
+O(\epsilon^2).
\]
This does not change the conformal class iff $K$ satisfies
(\ref{projective_field}) with $\rho=\d f$.
\end{proof}
Finally we give the main result of this Section, and establish a one--to--one correspondence between projective vector fields on $(\mf, [\nabla])$, and
Killing vector fields on the Einstein lift on $M$.
\begin{theo}
\label{theo_kill}
Let $K$ be a projective vector field on $(U, [\nabla])$, where $U\subset N$. 
Then
\be
\label{symp_kil}
{\mathcal K}:=K-\xi_j\frac{\partial K^j}{\partial x^i}\frac{\p}{\partial \xi_i}+
\rho_i\frac{\p}{\partial \xi_i}
\ee
is a Killing vector field for $g_{\nabla}$ which is symplectic with respect to the  symplectic form $\Omega_{\nabla}$.
Conversely, any Killing vector field for $g_{\nabla}$ is a lift 
(\ref{symp_kil}) from $\mf$ of some projective vector field.
\end{theo}
\noindent
\begin{proof} The integrability conditions for (\ref{projective_field})
are \cite{yano2} (note however that that our sign conventions for the Schouten tensor
differ from that in \cite{yano2}, so the sign of the RHS of (\ref{rho_integrability})
is opposite to what is given in \cite{yano2})
\be
\label{rho_integrability}
{\mathcal L}_K\Rho_{ij}=-\nabla_i\rho_j.
\ee
We shall also write ${\mathcal K}=\tilde{K}+ K_\rho$, where
$\tilde{K}$ is the complete lift (\ref{complete_lift}) 
and $K_\rho:=\rho_i\partial/\partial \xi_i$. Using~\eqref{eq:metricchoicecon} we compute
\begin{eqnarray*}
{\mathcal L}_{\mathcal K} g_{\nabla}&=& {\mathcal L}_{\widetilde{K}} h_{\nabla}+
 {\mathcal L}_{\widetilde{K}} \lambda\odot\lambda
+ {\mathcal L}_{K} \mathrm{Ric}^+(\nabla)+
 {\mathcal L}_{K_\rho} h_{\nabla}+
{\mathcal L}_{K_\rho} (\lambda\odot\lambda)\\
&=& -\xi_k\left({\mathcal L}_{{K}}\nabla\right)^k_{ij}  dx^i\odot dx^j
+0-(\nabla_i\rho_j) dx^i\odot dx^j
+ dx^i\odot d\rho_i\\
&&
- \Gamma_{ij}^k\rho_k dx^i \odot dx^j +
(\rho_idx^i)\odot (\xi_j dx^j)=0,
\end{eqnarray*}
where we have used (\ref{formula}), (\ref{projective_field})
and (\ref{rho_integrability}).

Now verify the symplectic condition
\begin{eqnarray*}
{\mathcal L}_{\mathcal K} \Omega_{\nabla}&=&
{\mathcal L}_{\widetilde{K}} (d\xi_i\wedge dx^i)
+{\mathcal L}_{K_\rho} (d\xi_i\wedge dx^i)
+{\mathcal L}_{K} (\Rho_{ij}dx^i\wedge dx^j)\\
&=&(d\rho_i\wedge dx^i- d\rho_i\wedge dx^i)=0
\end{eqnarray*}
as the complete lift $\widetilde{K}$ is Hamiltonian with respect to
$d\xi_i\wedge dx^i$, and we have used the skew part of the integrability conditions (\ref{rho_integrability}).

To prove the converse, consider a general vector field 
${\mathcal K}=K^i\partial/\partial x^i
+Q_i\partial/\partial \xi_i$ on $M$, and impose the Killing equations.
The $d\xi_i\odot d\xi_j$ components of these equations imply that $K^j=K^j(x^1, x^2)$.
The $d\xi_i\odot d x^j$ components yield the general form  (\ref{symp_kil}),
where $\rho_i$ are some unspecified functions on $N$. Finally 
the $dx^i\odot dx^j$ components imply that the vector field 
$K^i\p/\p x^i$ on $N$ is projective.
\end{proof}

\section{Local characterization of the metric}

In the previous section we have shown that the metric $g$ constructed on the canonical affine bundle of a projective surface $(N,[\nabla])$ is isometric to the metric 
\begin{equation}\label{eq:dmmetricchoiceaffine}
g_{\nabla}=h_{\nabla}+\lambda^2+\nu^*\mathrm{Ric}^+(\nabla)
\end{equation}
on the cotangent bundle $\nu : T^*N \to N$ of $N$, where $\nabla \in [\nabla]$ is any representative connection. The metric~\eqref{eq:dmmetricchoiceaffine} has previously appeard in~\cite{spanish} as a member of a one-parameter family $g_{\nabla,\Lambda}$ of split-signature metrics on $T^*N$ that one can associate to a torsion-free connection on $N$. The metrics take the form
\begin{equation}\label{eq:oneparamfamily}
g_{\nabla,\Lambda}=h_{\nabla}+\Lambda\,\lambda^2+\left(\frac{1}{\Lambda}\right)\nu^*\mathrm{Ric}^+(\nabla), 
\end{equation}
where $\Lambda$ is any non-zero real number. In particular, in~\cite{spanish} it is noted that the metrics $g_{\nabla,\Lambda}$ are anti-self-dual\footnote{self-dual with respect to the orientation convention of~\cite{spanish}.} and Einstein with scalar curvature $24\Lambda$, as can easily be verified by direct computation. Moreover, under the assumption that $\nabla$ is non-flat, the metrics $g_{\nabla,\Lambda}$ are locally characterized as the neutral signature four-dimensional type II Osserman metrics whose Jacobi operator have non-zero eigenvalues. We refer the reader to~\cite[Thm.~7.3]{spanish}
 for details. Here we provide another characterisation. Recall \cite{Gilkey,Chudecki,Walker_1} that
a distribution ${\mathcal D}\subset TM$ on a Riemannian manifold $(M,g)$ is called parallel if
${{}^g\nabla}_X Y\in\Gamma ({\mathcal D})$ if  
$Y\in\Gamma ({\mathcal D})$, where ${{}^g\nabla}$ is the Levi--Civita connection 
of $g$. 
Thus, if ${\mathcal D}$ is parallel, then
it is necessarily Frobenius integrable
as $[X, Y]={{}^g\nabla}_X Y - {{}^g\nabla}_Y X  \in\Gamma ({\mathcal D})$
if  $X, Y\in\Gamma ({\mathcal D})$.
\begin{theo}
\label{theo_asd}
Let $(M,g)$ be an ASD Einstein manifold with scalar curvature $24$ admitting a parallel 
ASD totally null distribution. Then $(M,g)$ is conformally flat, or it is locally isometric to $(T^*N,g_{\nabla})$ for some torsion-free connection $\nabla$ on $N$.
\end{theo}
\noindent
\begin{proof}
We shall rely on  two isomorphisms: $TM=\spp\otimes \spp'$, and 
${\Lambda^2}_-=\spp\odot \spp$. Let the ASD totally null distribution
correspond to an ASD two-form $\Theta\in\Gamma({\Lambda^2}_-)$, or equivalently to a section $\iota\in\Gamma(\spp)$. The Frobenius integrability conditions imply the local existence of two functions
$\xi_1$ and $\xi_2$ on $M$ such that $\mbox{Ker}(\Theta)=\mbox{span}\{\partial/\partial \xi_1, \partial/\partial \xi_2\}$. We can rescale $\iota$ so that the corresponding two--form is 
closed, and proportional to $dx^1\wedge dx^2$ for some functions $(x^1, x^2)$
which are constant on each $\beta$--surface in the two parameter family.
The functions $(\xi_1, \xi_2)$ are then the coordinates on the $\beta$--surface.
The corresponding metric takes the form
\[
g=d\xi_i\odot dx^i+\Theta_{ij}(x,\xi)dx^i\odot dx^j
\]
for some symmetric two-by-two matrix $\Theta$. The anti--self--duality 
condition on the Weyl tensor forces the components of $\Theta$ to be at most 
cubic in $(\xi_1, \xi_2)$, with some additional algebraic relations between the components. Imposing the Einstein  
condition  gives 
$$
\Theta_{ij}=\xi_i\xi_j+\Rho_{ji}-\Gamma_{ij}^k\xi_k,
$$ 
where the functions $\Gamma_{ij}^k$ do not depend on the coordinates $\xi_1,\xi_2$ and are otherwise arbitrary. Finally, the functions $\Rho_{ij}$ are determined by (\ref{formula_for_rho}). Comparing with the coordinate expression~\eqref{coordexpmetsymp} proves the claim.
\end{proof}

\begin{rmk}
If $\nabla$ is a torsion-free connection on $N$ with skew-symmetric Ricci tensor, then~\eqref{eq:oneparamfamily} simplifies to become
$$
g_{\nabla,\Lambda}=h_{\nabla}+\Lambda\lambda^2. 
$$
In particular, the limit $\Lambda\to 0$ is well-defined and hence the metric $g_{\nabla}$ can be deformed to a Ricci-flat anti-self-dual metric $g_{\nabla,0}=h_{\nabla}$ which appeared in~\cite{derdzinski}. 
\end{rmk}
\begin{rmk}
Note that if we correspondingly define a~\textit{charged symplectic form}\footnote{This terminology is motivated by
the Hamiltonian description of a charged particle moving on a manifold, where the canonical symplectic structure on the cotangent bundle needs to be modified by a pull-back of a closed two-form (magnetic field) from the base manifold. In our case the two-form is the skew-symmetric part
of the Schouten tensor, and the inverse of the Ricci scalar plays 
a role of electric charge. This magnetic term can always be set to zero by an appropriate choice of a connection in a projective class - here we find it convenient not to make any choices at this stage.}
$$
\Omega_{\nabla,\Lambda}=d\lambda+\left(\frac{1}{3\Lambda}\right)\nu^*\mathrm{Ric}^{-}(\nabla),
$$
then the pair $(g_{\nabla,\Lambda},\Omega_{\nabla,\Lambda})$ defines a bi-Lagrangian structure on $T^*N$ for every $\Lambda\neq 0$. The symplectic form $\Omega_{\nabla,\Lambda}$ is ASD with respect to our choice of orientation and the metric \eqref{eq:oneparamfamily}. Moreover, denoting by ${{}^g\nabla}$ the Levi-Civita connection of the metric $g_{\nabla,\Lambda}$, we obtain
$$
{{}^g\nabla}\,\Omega_{\nabla,\Lambda}=4L
$$
where $L$ is the pull--back to $M$ of the Liouville curvature
$\epsilon^{ij}\nabla_{i}\Rho_{jk}dx^k\otimes (dx^1\wedge dx^2)$ of $[\nabla]$, 
which vanishes if and only if $\nabla$ is projectively flat.
\end{rmk}

\begin{rmk}
Straightforward calculations show that Theorem~\ref{theo_kill} carries over to the case $(g_{\nabla,\Lambda},\Omega_{\nabla,\Lambda})$ with respect to the lift
$$
{\mathcal K}:=K-\xi_j\frac{\partial K^j}{\partial x^i}\frac{\p}{\partial \xi_i}+\frac{1}{\Lambda}
\rho_i\frac{\p}{\partial \xi_i}.
$$
\end{rmk}

\begin{rmk}
The existence of a neutral metric $g$ with a two--plane distribution 
imposes
topological restrictions on $M$. If $M$ is compact then \cite{atiyah, HH}
\[
\chi[M] \equiv 0 \ \mbox{mod} \ 2, \ \ \ \ \ \chi[M] \equiv \tau[M] \
\mbox{mod} \ 4,
\]
where $\tau[M]$ and $\chi[M]$ are the signature and Euler characteristic
respectively. C.~LeBrun pointed out to the authors the following argument which shows that a stronger statement is true in the case where the two-plane distribution $\mathcal{D}$ is totally null with respect to $g$.\footnote{Private communication, March 2016.} We may assume that $\mathcal{D}$ is the graph of an isomorphism $\eee \to \eee^{\prime}$, where $TM=\eee\oplus \eee^{\prime}$ is an orthogonal decomposition into time-like and space-like sub-bundles with respect to some chosen background metric $h$ on $M$. After possibly passing to a double cover we can assume $\eee$ and $\eee^{\prime}$ to be orientable. Moreover, we may fix orientations so that the isomorphism $\eee \to \eee^{\prime}$ is orientation reversing, thus equipping $M$ with an orientation so that $\mathcal{D}$ is anti-self-dual. By rotating clockwise in $\eee$ and $\eee^{\prime}$ with respect to $h$, we obtain an almost complex structure on $M$ such that $\eee$ becomes a complex line sub-bundle $L$, and so that $\eee^{\prime}$ becomes its dual bundle $L^*$. Consequently, $M$ admits an almost complex structure $J$ such that the canonical bundle of $(M,J)$ is trivial. After possibly passing to a double cover it therefore follows that $M$ is oriented and spin and -- assuming $M$ is compact -- that
\begin{equation}\label{eq:topcond}
2\chi[M]+3\tau[M]=0.
\end{equation}
Note that fixing the orientation so that $\mathcal{D}$ is self-dual leads to a sign change in~\eqref{eq:topcond} as $\tau$ changes sign when reversing the orientation whereas $\chi$ does not. Also, note that the existence of $\mathcal{D}$ forces $M$ to be orientable hence~\eqref{eq:topcond} still holds true (assuming our choice of orientation) without passing to the cover as $\chi$ and $\tau$ are both doubled when passing to a double cover. 
\end{rmk}

\section{Gauge theory of Tractor Connection}
In this Section we shall present a gauge--theoretic construction of the 
metric (\ref{main_metric}). We shall introduce a projectively invariant
equation on a connection, and a pair of Higgs fields on an auxilary vector
bundle $E\rightarrow N$.  In the special case when $E$ is a rank--3 cotractor
bundle (see \S\ref{main_tractor_section}) and the gauge group is $SL(3, \R)$,
the horizontal lifts of the geodesic spray of $\nabla$ and the Higgs field
will give rise to an integrable  $\alpha$--plane (twistor) distribution on 
$TM$, where $M=\PP(E)$ with a projective line removed from each fiber.

Let $(\mf, [\nabla])$ be a projective structure on a surface, and let
$E\rightarrow \mf$ be a vector bundle  with
$\mathfrak{g}$--connection $A$, where $\mathfrak{g}$ is some Lie algebra. 
Let $\phi$ be a one-form on $N$, called the Higgs pair, 
with values  in the Lie algebra $\mathfrak{g}$. 
In an open set $U\subset \mf$ we shall write 
$\phi=\phi_i dx^i$ and regard $\phi$ and $A$ as $\mathfrak{g}$ valued
one-forms on $N$ transforming as
\begin{eqnarray*}
&A&\longrightarrow \gamma A\gamma^{-1}-d\gamma\;\gamma^{-1}\\
&\phi&\longrightarrow\gamma\phi\gamma^{-1}
\end{eqnarray*}
under the gauge transformations. Here $\gamma:N\rightarrow G$, and
$G$ is the gauge group with the Lie algebra $\mathfrak{g}$.

For any  chosen connection
$\nabla\in [\nabla]$ in 
the projective class consider the system of equations
\be
\label{cal_2}
D_{(i}\phi_{j)}=0,
\ee 
where 
\[
D_i\phi_j:=\partial_i\phi_j-\Gamma_{ij}^k\phi_k-[A_i, \phi_j].
\]
In \cite{Cal1,Cal2} these equations appear in a slightly different setup,
when $A$ is a connection on a principal (rather than a vector) bundle.
While our construction below is self--contained, and 
does not rely on the results of 
\cite{Cal1,Cal2}, we shall nevertheless refer to  (\ref{cal_2})
as the Calderbank equations.
\subsection{The Calderbank equations.}
An equivalent way to formulate (\ref{cal_2})
 is to say that 
the Higgs pair 
is constant along the charged geodesic
spray on $T\mf$, i.e.
\be
\label{cal_1}
{\bf \Theta}^A(\phi):=\Big(\pi^i\frac{\partial}{\partial x^i}-\Gamma_{ij}^k\pi^i\pi^j\frac{\partial}{\partial \pi^k}\Big) (\phi)-[A, \phi]=0,
\ee
where  $\pi^i$ are coordinates on the fibres of $T\mf$, and 
$\phi=\phi_i\pi^i$ and  $A=A_i\pi^i$ are $\mathfrak{g}$--valued linear
functions on $T\mf$.
The equations (\ref{cal_2}) do not depend on the choice of the connection 
$\nabla$ in 
the  projective class if the Higgs field $\phi$ has projective weight $2$.

In \S\ref{main_tractor_section}  we shall show how the Calderbank equations with the gauge group
$\mathrm{SL}(3, \R)$ -- regarded as a subgroup of the group of diffeomorphisms of
$\RP^2$ -- leads to the neutral signature anti--self--dual Einstein metric (\ref{coordexpmetsymp}).
We shall first list some other (implicit) occurrences
of these equations for other gauge groups.
\subsubsection{Null reductions of anti-self-dual Yang--Mills equations.} 
If the projective structure is flat, then (\ref{cal_2}) is 
the symmetry reduction of the anti-self-dual Yang--Mills
(ASDYM) equation on $\R^{2, 2}$ by two null translations and such 
that the $(2, 2)$ 
metric $g$ restricted to the two--dimensional space of orbits $\mf=\R^2$
is totally isotropic, and the bi-vector generated by the null translations is anti-self-dual.

To see it, consider a $\mathfrak{g}$--valued connection one--form $A$ on 
$\R^{2, 2}$, and set
$F=dA+A\wedge A$.
In local coordinates adapted to $\R^{2, 2}=T\mf$ with $x^i$ the coordinates 
on $\mf$,
the null isometries are $\partial/\partial \xi_i$, and
the metric is
\[
g=dx^1d\xi_1+dx^2d\xi_2.
\]
Choose an orientation on $\R^{2, 2}$ such that the two--form $dx^1\wedge dx^2$ is ASD.
Defining two Higgs fields $\phi_1=\partial/\partial \xi_2\hook A, \phi_2=\partial/\partial \xi_1\hook A$, 
the ASDYM equations $F=-*F$ yield \cite{MW}
\be
\label{null_ASD}
D_1 \phi_1=0, \quad D_2 \phi_2=0, \quad D_1 \phi_2+D_2 \phi_1=0,
\ee
where $D=d+A_1dx^1+A_2dx^2$ is a covariant derivative on $\mf$ induced by $A$.
In \cite{tafel} these equations have been solved completely for the gauge 
group $\mathrm{SL}(2)$.
\subsubsection{Prolongation of the Calderbank equations}
Instead of regarding both the connection and the Higgs pair as unknowns, we
shall assume that the connection is given and consider  (\ref{null_ASD})
as a system of PDEs for the Higgs pair.
To determine all derivatives of the Higgs pair in (\ref{null_ASD})
we prolong the system once, and define $\mu$ by the equation
\[
D_i\phi_j=\frac{1}{2}\mu\epsilon_{ij},
\]
where $\epsilon=dx^1\wedge dx^2$ is the parallel volume form of $\nabla\in [\nabla]$.
Commuting the covariant derivatives now leads to a closed 
system and therefore a connection
\[
D_i\left(\begin{array}{c}
\phi_j\\ 
\mu
\end{array} \right)=\left(\begin{array}{c}
\frac{1}{2}\mu\epsilon_{ij}\\ 
2[\phi_i, {\mathcal F}]
\end{array} \right),
\]
where ${\mathcal F}=[D_1, D_2]$ is the $\mathfrak{g}$--valued curvature of the connection 
$A$. The system is now closed. Commuting the covariant derivatives on $\mu$ leads to an integrability condition
\[
[{\mathcal F}, \mu]-2D_1[\phi_2, {\mathcal F}]+2D_2[\phi_1, {\mathcal F}]=0.
\]
\subsubsection{Killing equations.} 
If the connection $A$ is flat, and $\mathfrak{g}=\R$ then
the Calderbank equations become the projectively invariant
Killing equations.
\subsubsection{Anti-self-dual conformal structures with null conformal Killing vectors.}
Let $\mathfrak{g}$ be a subalgebra of the infinite dimensional Lie algebra of vector fields $\mathfrak{diff}(\Sigma)$ on a surface $\Sigma$
consisting of those elements of $\mathfrak{diff}(\Sigma)$ which commute with a 
fixed vector field $K$ on $\Sigma$. Let $M\rightarrow\mf$ be a surface bundle over $\mf$, with two dimensional fibres $\Sigma$.
In this case the Calderbank equations
are solvable by quadrature and the two-dimensional distribution 
\be
\label{twistor_dist}
{\mathcal D}=\{
{\bf \Theta}^A:=\pi^i\frac{\p}{\partial x^i}-\Gamma_{ij}^k\pi^i\pi^j\frac{\p}{\partial \pi^k}
-A_i(x)\pi^i, \; \phi=\pi^i\phi_i\}
\ee
spanning an $\RP^1$ worth of null self--dual surfaces ($\alpha$--surfaces) through
each point of $M$ is the twistor distribution
for the most general ASD $(2, 2)$ conformal 
structure
which admits a null conformal Killing vector $K$ \cite{DW, Cal1, nakata}.
\subsubsection{The Patterson-Walker Riemannian extension}
The conformal structure resulting from the distribution
(\ref{twistor_dist}) is a generalisation of the Patterson-Walker
lift \cite{walker,Gilkey}.  To recover the Patterson-Walker metric 
\be
g=d\xi_i\odot dx^i-\Gamma_{ij}^k\;\xi_k\; dx^i\odot dx^j,
\ee
take
the gauge algebra $\mathfrak{g}=\mathfrak{gl}(2,\R)$ which generates linear
transformations of $\Sigma=\R^2$. If the coordinates on $\Sigma$ 
are $(\xi_1, \xi_2)$,
the elements of $\mathfrak{gl}(2,\R)$ are vector fields of the 
form 
$
{{\bf t}_i}^j= \xi_i\frac{\p}{\partial \xi_j}.
$
Taking the 
connection $A$ and the  Higgs field given 
by
\[
A_i=-\Gamma_{ij}^k\;\xi_k\frac{\p}{\partial \xi_j}, \quad \phi_i=(b^k\xi_k)\epsilon_{ij}\frac{\p}{\partial \xi_j}
\]
where $b^k$ is a non--zero constant
leads to an integrable distribution (\ref{twistor_dist}), as
then
\begin{eqnarray*}
[{\bf \Theta}^A, \phi] &=&\pi^ib^j
\Big( {\Gamma_{jk}}^k\xi_i+ {\Gamma_{ij}}^k\xi_k\Big)
\epsilon_{lm}\pi^l\frac{\p}{\partial \xi_m}\\
&=&0 \quad(\mbox{mod}\; \phi).
\end{eqnarray*}
The resulting
metric (\ref{walker}) on $M=T\mf$ is then uniquely determined by the condition that  the integral two--surfaces of ${\mathcal D}$ in $T\mf\times\RP^1$ project down to self-dual totally null surfaces 
on $T\mf$. 
The generators of the gauge group
satisfy
$
[{\bf v}, K]={\bf v},
$
where $K=\xi_1\partial/\partial \xi_1+\xi_2\partial/\partial \xi_2$ is a 
conformal null Killing vector
of (\ref{walker}).

We shall end this subsection by clarifying  the connection between
projective changes of $\nabla\subset[\nabla]$, and conformal 
rescalings of the metric $g$ on $T^*\mf$. 
We shall 
restrict our discussion to {\em special connections} in 
$[\nabla]$ which preserve some volume.
Consider the effect of transformation (\ref{connection_change}) with $\Upsilon_i=\nabla_i f$, together with rescaling the fibers of $T\mf\rightarrow \mf$
\[
\xi_i\rightarrow \hat{\xi}_i=e^{2f}\xi_i
\]
on the Patterson--Walker lift\footnote{In \cite{DT10} (see also 
\cite{BDE,hamerl,hamerl2} for other applications of this lift) 
it was proven that a `similar' metric
\be
\label{walker_dt}
g=d\xi_i\odot dx^i-\Pi_{ij}^k\;\xi_k\; dx^i\odot dx^j,
\ee
constructed out of the 
Thomas symbols $\Pi_{ij}^k=\Gamma_{ij}^k-\frac{1}{3}\Gamma_{il}^l\delta^k_j
-\frac{1}{3}\Gamma_{jl}^l\delta^k_i$
is anti--self--dual  and null--K\"ahler  (with ASD null--K\"ahler two--form) for any choice of $\Gamma_{ij}^k$.
The Patterson--Walker lift (\ref{walker}) is conformally equivalent 
(up to a diffeomorphism) to the projective
Patterson--Walker lift (\ref{walker_dt})
 only if
$\Gamma_{ij}^j=\nabla_i F$ for some function $F$ on $N$.
}  (\ref{walker}). A straightforward calculation yields
\[
\hat{g}=e^{2f} g.
\]
Thus conformal scales on $T\mf$ correspond to projective scales on $\mf$.
\subsection{Tractor connection and ASD Einstein metrics}
\label{main_tractor_section}
In this Section we shall consider the Calderbank equations, where
the gauge group is $\mathrm{SL}(3, \R)$, and $E$ is the 
standard cotractor bundle
for the projective structure $[\nabla]$. Recall the Cartan bundle
$\projb$ from section (\ref{cartgeomsec}).
We may think of the left action of $H\subset \mathrm{SL}(3,\R)$ on $\R_3$ by matrix multiplication as a (linear) $H$-representation and consequently, we obtain an associated rank-$3$ vector bundle $E$ for every projective surface $(\mf,[\nabla])$. The vector bundle $E$ is commonly referred to as the~\textit{cotractor bundle} of $(\mf,[\nabla])$. Interest in $E$ stems from the fact that it comes canonically equipped with an $\mathrm{SL}(3,\R)$ connection which is flat if and only if $(\mf,[\nabla])$ is, see~\cite{BEG}.


Let $\mathcal{E}(1)$ be the line bundle of projective densities of weight 1.
Consider a rank-three vector bundle $E={\mathcal E}(1)\oplus
(T^*\mf\otimes{\mathcal E}(1))$ over  $\mf$ with connection \cite{BEG}
\be
\label{tractor_con}
{\quad{\mathcal{D}}_i \left(\begin{array}{c}
\sigma\\ 
\mu_j
\end{array} \right)= 
\left(\begin{array}{c} \nabla_i\sigma-\mu_i \\ 
\nabla_i\mu_j+\Rho_{ij}\sigma
\end{array} \right),}
\ee
where $\Rho_{ij}$ is the (not necessarily symmetric) Schouten tensor of projective geometry. The splitting of the cotractor bundle depends on a choice of a connection $\nabla$ in the projective class $[\nabla]$, and 
under (\ref{connection_change})
changes according to 
\be
\label{tractor_change}
\left(\begin{array}{c}
\hat{\sigma}\\ 
\hat{\mu}_j
\end{array} \right)=
\left(\begin{array}{c}
\sigma\\ 
\mu_j+\Upsilon_j\sigma
\end{array} \right).
\ee
Using the tractor indices $\alpha, \beta, \dots =0, 1, 2$ we can rewrite the connection (\ref{tractor_con}) in terms of its Christoffel symbols
$\gamma_{i\alpha}^\beta$  as
\[
\gamma_{i0}^0=0, \quad \gamma_{i0}^j=\delta_i^j,\quad
\gamma_{ij}^k=\Gamma_{ij}^k, \quad \gamma_{ij}^0=-\Rho_{ij}.
\]
The vector fields
\[
{{\bf t}_\alpha}^\beta=\psi_\alpha\frac{\p}{\partial \psi_\beta}
\]
generate the linear action of $\mathrm{GL}(3, \R)$ on the fibres of $E$.
These generators descend to eight vector fields
(which we shall also denote ${{\bf t}_\alpha}^\beta$) which generate
the action of $\mathrm{SL}(3, \R)$ on the fibres of the projective cotractor
bundle $\PP(E)$ which is a quotient of $E$ by the Euler vector field
$\sum_{\alpha=0}^2 {{\bf t}_\alpha}^\alpha$. Setting $\xi_i=\psi_i/\psi_0$ yields
\[
{{\bf t}_i}^j=\xi_i\frac{\p}{\partial \xi_j}, \quad  
{{\bf t}_i}^0=-\xi_i\xi_j\frac{\p}{\partial \xi_j}, \quad
{{\bf t}_0}^i=\frac{\p}{\partial \xi_i}, \quad {{\bf t}_0}^0=-\xi_j\frac{\p}{\partial \xi_j}.
\]
Consider the Calderbank equations with the gauge group $\mathrm{SL}(3, \R)\subset\mbox{Diff}(\RP^2)$, where
the connection is given by a vector--valued one-form
\[
A=A_idx^i=-\gamma_{i\beta}^{\alpha} dx^i\otimes\;{{\bf t}_\alpha}^\beta
\]
so that
\[
A_i=(\Rho_{ij}+\xi_i\xi_j  -\Gamma_{ij}^k\;\xi_k)\frac{\p}{\partial \xi_j}.
\]
The Calderbank equations are solved by the Higgs pair
\[
\phi_i=\epsilon_{ij}\frac{\p}{\partial \xi_j}.
\] 
Let $M$ be a complement of a projective line in the total space of the bundle $\PP(E)$.
The corresponding contravariant metric on $M$ is constructed by
demanding that the leaves of the rank-2 distribution (\ref{twistor_dist})
${\mathcal D}\subset T(M\times\RP^1)$ project down to self-dual
two-surfaces on $M$. This gives
$
\epsilon^{ij}(\partial/\partial x^i -A_i)\odot \phi_j,
$
or, in the covariant form,
\be
\label{einstein_lift}
g=(d\xi_i-(\Gamma_{ij}^k \xi_k -\xi_i\xi_j-\Rho_{ji})dx^j)\odot dx^i,
\ee
so that we have recovered the metric of the bi-La\-grang\-ian structure
(\ref{coordexpmetsymp}).
\begin{theo}
\label{theo_char}
Formula (\ref{einstein_lift}) defines a metric which does not depend on a
choice of a connection in a projective class.
\label{theo_einstein}
\end{theo}
\noindent
\begin{proof}
If we change the connection in the
projective class using (\ref{connection_change}) then the Schouten tensor changes by
(\ref{schouten_change}). To establish the invariance of (\ref{einstein_lift}) we translate the fibre coordinates according to
\[
\hat{\xi}_i=\xi_i+\Upsilon_i
\]
in agreement with (\ref{tractor_change}).
Then
\begin{eqnarray*}
&&(d\hat{\xi}_i-({\hat{\Gamma}_{ij}}^k \hat{\xi}_k-\hat{\xi}_i\hat{\xi}_j-\hat{\Rho}_{ji})dx^j)\odot dx^i
= d\xi_i\odot dx^i +\\
&&\Big(\xi_{(j} \Upsilon_{i)} -  \Gamma_{ij}^k\xi_k
-\xi_i\Upsilon_j-\xi_j\Upsilon_i- \Gamma_{ij}^k\Upsilon_k-2\Upsilon_i\Upsilon_j
+\xi_i\xi_j+\xi_i\Upsilon_j+\xi_j\Upsilon_i+\Upsilon_i\Upsilon_j
\\
&&+\Rho_{ji}-\nabla_{(j}\Upsilon_{i)}+\Upsilon_i\Upsilon_j\Big)dx^i\odot dx^j\\
&&=
\left(d \xi_i-\left(\Gamma_{ij}^k \xi_k
-\xi_i\xi_j-{\Rho}_{ji}\right)dx^j\right)\odot dx^j.
\end{eqnarray*}
\end{proof}
The metric is anti-self-dual, and Einstein
with scalar curvature equal to $24$. The anti-self-duality is a consequence of the fact
that the connection $A$ and the Higgs field $\phi_i\pi^i$ satisfy the 
Calderbank equations \cite{Cal1}. 

\section{Examples}
\subsection{Homogeneous model $M=\mathrm{SL}(3, \R)\mm\mathrm{GL}(2, \R)$.}
\label{sec_homo}
Consider the flat projective structure on $(\mf=\RP^2, [\nabla])$, and choose
${\Gamma_{ij}}^k=0$. The resulting four manifold is the complement of 
an $\RP^1$ sub-bundle in the projective cotractor bundle of $\RP^2$ which can be identified with with $M=\mathrm{SL}(3, \R)\mm\mathrm{GL}(2, \R)$. We shall establish this result in arbitrary dimension. Consider $\mf=\RP^n$, with
its flat projective structure, and an $\mathrm{SL}(n+1)$ action on the projective
cotractor bundle $\PP(E)$ minus the diagonal   
\[
\mathrm{SL}(n+1):\R^{n+1}\times \R_{n+1}\mm\Delta\longrightarrow \R^{n+1}\times \R_{n+1}\mm\Delta
\]
where the `diagonal' $\Delta$ consists of all incident pairs of vectors $[v]\in \R^{n+1}$ and forms $[f]\in \R_{n+1}$ s.t. the corresponding point $v\in \RP^n$
belongs to  the hyperplane $f\in \RP_n$. This action is simply 
$(v, f)\rightarrow(Av, fA^{-1})$. It  is transitive, and clearly a subgroup stabilising a pair (point,  hyperplane) is $\mathrm{GL}(n)$ which sits in $\mathrm{SL}(n+1)$ as a lower diagonal
block.

To finish the proof we need to argue that $\R^{n+1}\times \R_{n+1}\mm\Delta$ projects down to a complement of an $\RP_{n-1}$ sub-bundle in $\PP(E)$. This
sub-bundle is just $\PP(T^*\mf)$ and it has an injection into $\PP(E)$ 
given by $f\rightarrow (0, f)$. A point in $\mf$ with homogeneous coordinates $[1, 0, ..., 0]$
(corresponding to our  choice of an affine chart) is not incident with
any cotractor in $\PP(E)/\RP_{n-1}$, so removing a diagonal is equivalent to
looking at the complement of this sub-bundle.

The Einstein metric
(\ref{einstein_lift}) on this manifold admits a Kerr-Schild form
\be
\label{sl3gl2}
g=d\xi_i\odot dx^i +\Lambda (\xi_jdx^j)^2
\ee
with eight dimensional isometry group $\mathrm{SL}(3, \R)$ \cite{CDT13}
in agreement with Theorem \ref{theo_kill}).
This metric is a neutral signature analog of the Fubini--Study metric
on $\CP^2$. Both metrics arise as different real forms
of $\mathrm{SL}(3, \C)/\mathrm{GL}(2, \C)$. The limit $\Lambda=0$ 
in (\ref{sl3gl2}) gives the flat metric.
\subsection{Ricci--flat limits.} Motivated by the previous example
let us now consider the general case of projective structures which admit a connection with skew-symmetric
Schouten tensor. In this case one can always choose local coordinates
on $N$ and a connection $\nabla\in [\nabla]$ such that
\cite{Wong}
\[
\Gamma_{11}^1=-\frac{\partial f}{\partial x^1}, \quad 
\Gamma_{22}^2=\frac{\partial f}{\partial x^2},
\]
where $f:\mf\longrightarrow \R$ is an arbitrary function, and all other components of $\nabla$
vanish\footnote{An alternative characterisation of the 
corresponding projective structures is that they arise from second-order ODEs point equivalent to derivatives of first order ODEs \cite{DW}.
These projective structures where further characterised in 
\cite{randal} and \cite{krynski}.}. In this case
\[
\Rho=\frac{1}{3}\frac{\p^2 f}{\partial x^1\partial x^2} dx^1\wedge dx^2,
\]
and the metric is given by
\be
\label{skew_metrics}
g=d\xi_i\odot dx^i +\xi_1\frac{\partial f}{\partial x^1} (dx^1)^2
-\xi_2\frac{\partial f}{\partial x^2} (dx^2)^2+
\Lambda (\xi_jdx^j)^2.
\ee
Setting $\Lambda=0$ gives an ASD Ricci-flat 
metric which has a form of the Patterson--Walker lift (\ref{walker}) 
and has  appeared in the work of Derdzinski \cite{derdzinski}.
\subsection{Cohomogeneity--one examples}
The dimension of the Lie algebra $\mathfrak{g}$ of projective vector fields for a given  projective structure on a surface $N$ can be $8, 3, 2, 1$ or $0$
(see \cite{Lie}, and also \cite{romanovskii,bryant_mano,dumitriescu}). If 
the dimension is maximal and equal to 8 then $\mathfrak{g}=\mathfrak{sl}(3, \R)$, and
the projective structure is flat. We have shown that in this case the resulting
metric (\ref{main_metric}) is given by (\ref{sl3gl2}), and admits $8$ Killing 
vectors in agreement with Theorem \ref{theo_kill}. We shall now consider the 
submaximal case, where $\mathfrak{g}=\mathfrak{sl}(2, \R)$. There are two one--parameter families of non--flat projective structures with this symmetry. Their unparametrised geodesics are integral curves of a second order ODE
\[
y''=c(xy'-y)^3,
\]
where $c\neq 0$. We compare this to the general second order ODE defining a projective structure (see e.g. \cite{BDE})
\be
\label{projective_ode}
y''=\Gamma_{22}^1 (y')^3+(2\Gamma_{12}^1-\Gamma_{22}^2) (y')^2+(\Gamma_{11}^1-
2\Gamma_{12}^2)y'-\Gamma_{11}^2,
\ee
and chose the representative connection $\nabla$ by
\[
\Gamma_{11}^1= -\Gamma_{12}^2=-\Gamma_{21}^2 =cxy^2, \quad 
\Gamma_{22}^2=- \Gamma_{21}^1= -\Gamma_{12}^1= cx^2y, \quad \Gamma_{22}^1=cx^3, \quad \Gamma_{11}^2=cy^3.
\]
The corresponding ASD Einstein metric (\ref{main_metric}) is
\be
\label{sl_metric}
g=d\xi_i \odot dx^i +\Lambda(\xi_i dx^i)^2+\frac{4c}{\Lambda}(x^2dx^1-x^1dx^2)^2-\Gamma_{ij}^kdx^i\odot dx^j,
\ee
where $x^i=(x, y)$. This metric admits a three--dimensional isometry group
$SL(2, \R)$
generated by left--invariant vector fields $K_\alpha, \alpha=1, 2, 3$ given by
\[
K_1=x^1\frac{\p}{\p x^1}-x^2\frac{\p}{\p x^2}-
\xi_1\frac{\p}{\p \xi_1}+\xi_2\frac{\p}{\p \xi_2}, \quad
K_2=2x^1\frac{\p}{\p x^2}-2\xi_2\frac{\p}{\p \xi_1}, \quad
K_3= 2\xi_1\frac{\p}{\p \xi_2}- 2x^2\frac{\p}{\p x^1},
\]
and acting on $M=\R\times SL(2, \R)$ with three--dimensional orbits. We shall
use an invariant coordinate $r$ given by
 $r^2\equiv {(x^1\xi_1+x^2\xi_2)}$ 
which is constant on the orbits. Let $\sigma^{\alpha}$ be right--invariant one--forms
on $SL(2, \R)$ such that
\[
{\mathcal{L}}_{K_\alpha} \sigma^\beta=0, \quad \forall \alpha, \beta,\quad
\mbox{and}\quad
d\sigma^1+2\sigma^2\wedge\sigma^3=0, \quad
d\sigma^2+\sigma^2\wedge\sigma^1=0, \quad
d\sigma^3-\sigma^3\wedge\sigma^1=0.
\]
There is some freedom, measured by functions of $r$, in choosing these one--forms. 
If we chose $\Lambda<0$, and take
\[
\sigma^1= \frac{\xi_idx^i-x^id\xi_i}{r^2}+ \frac{2\Lambda rdr}{\Lambda r^2-1}, 
\quad \sigma^2=\frac{\Lambda r^2-1}{r^2}(x^1 dx^2-x^2dx^1),
\quad  \sigma^3=  
\frac{\xi_1d\xi_2-\xi_2 d\xi_1}{r^2(\Lambda r^2-1)}
\] 
then
the metric (\ref{sl_metric}) takes the form
\be
g=\label{homogeneous_metric}
\frac{dr^2}{1-\Lambda r^2}-\frac{1}{4}r^2(1-\Lambda r^2)(\sigma^1)^2-
\frac{c}{\Lambda}\frac{(\Lambda r^2-4)r^4}{(\Lambda r^2-1)^2}(\sigma^2)^2 +r^2\sigma^2\odot\sigma^3, \quad \Lambda<0.
\ee
Note that (\ref{homogeneous_metric}) is non--diagonal in the basis defined by the right--invariant one--form on $SL(2, \R)$. This is only possible in neutral
signature: All cohomogeneity one Einstein metrics in Riemannian signature
can be diagonalised \cite{DT17}.

The  metric (\ref{homogeneous_metric}) appears to be singular when $r=0$, but calculating the invariant 
norm of the Weyl curvature we find $|C|^2=96\Lambda^2$, which is regular. In fact near $r=0$ the metric
(\ref{homogeneous_metric}) approaches the space of constant curvature which is a 
neutral signature analogue of the hyperbolic space. To exhibit this space
in a standard form we neglect the small terms involving $r^4$, 
and set $r=2R/(1+\Lambda R^2)$. Then, near $R=0$,  the metric 
(\ref{homogeneous_metric}) becomes
\[
g\sim \frac{4}{(1+\Lambda R^2)^2}\Big( dR^2 -\frac{R^2}{4}\Big( (\sigma^1)^2-4\sigma^2\odot\sigma^3\Big)\Big).
\]
To this end, we note a curious Ricci--flat limit of (\ref{homogeneous_metric}).
Setting $c=m\Lambda$, and taking the limit $\Lambda\rightarrow 0$ yields
a Ricci--flat metric with 9--dimensional group of conformal isometries
\[
g=d\xi_i\odot dx^i+4m(x^2dx^1-x^1dx^2)^2.
\]
This is a submaximal  metric of neutral 
signature \cite{CDT13, the}: if the dimension of the conformal
isometry algebra $\mathfrak{g}$ exceeds 9, then $\mathfrak{g}=\mathfrak{sl}(4, \R)$,
and the metric is conformally flat.

\appendix

\section{The construction for higher dimensions}

Of course, the definition of a projective structure makes sense in higher dimensions as well and hence it is natural to ask if the construction described in the main body of this article carries over to higher dimensions. Here we briefly show that this is indeed the case.  

As usual, let $\mathrm{PGL}(n+1,\R)$ denote the quotient of the general linear group $\mathrm{GL}(n+1,\R)$ by its center $Z$, so that
$$
\mathrm{PGL}(n+1,\R)\simeq \left\{ \begin{array}{cl} \mathrm{SL}(n+1,\R) &n\;\text{even},\\\mathrm{SL}_{\pm}(n+1,\R)/\{\pm\mathrm{I}_{n+1}\}& n\;\text{odd},\end{array}\right.
$$
where $\mathrm{SL}_{\pm}(n+1,\R)$ denotes the group of real $(n+1)$-by-$(n+1)$ matrices with determinant $\pm 1$.  

The projective linear group acts from the left on $\mathbb{RP}^n=\left(\R^{n+1}\setminus\{0\}\right)/\mathbb{R}^*$ by matrix multiplication. The stabiliser subgroup of the line spanned by ${}^t(1\;0\;\ldots\;0)$ will be denoted by $G\subset \mathrm{PGL}(n+1,\R)$. The elements of $G$ are matrices of the form
$$
\begin{pmatrix} \det a^{-1} & b \\ 0 & a\end{pmatrix}
$$
for $n$ even and
$$
\left[\begin{array}{cc} \pm\det a^{-1} & b \\ 0 & a\end{array}\right]
$$
for $n$ odd, where $b\in \R_n$ and $a \in \mathrm{GL}(n,\R)$. Here, the square brackets indicate that the matrix is only well defined up to an overall sign.

Cartan's construction carries over to higher dimensions so that we canonically obtain a Cartan geometry $(\pi : \projb \to N,\theta)$ of type $(\mathrm{PGL}(n+1,\R),G)$ for every projective structure $[\nabla]$ on a smooth $n$-manifold $N$. Again, we write
$$
\theta=\left(\begin{array}{cc} -\tr \phi & \eta \\ \omega & \phi\end{array}\right)
$$
for an $\R_n$-valued $1$-form $\eta$, an $\R^n$-valued $1$-form $\omega$ and a $\mathfrak{gl}(n,\R)$-valued $1$-form $\phi$. The curvature $2$-form $\Theta$ satisfies
$$
\Theta=\d \theta+\theta\wedge\theta=\left(\begin{array}{cc} 0 & L(\omega\wedge\omega) \\ 0 & W(\omega\wedge\omega)\end{array}\right),
$$
for smooth curvature functions 
\[L : \projb \to \mathrm{Hom}\left(\R^n\wedge\R^n,\R_n\right)
\] and 
$$
W : \projb \to \mathrm{Hom}\left(\R^n\wedge\R^n,\R_n\otimes \R^n\right).
$$ 
Note that the function $W$ represents the~\textit{Weyl projective curvature tensor} of $(N,[\nabla])$ and that we have the Bianchi-identity
$$
\d \Theta=\Theta\wedge\theta-\theta\wedge\Theta,
$$
the algebraic part of which reads
\begin{equation}\label{algbianchi:again}
0=L(\omega\wedge\omega)\wedge\omega\quad \text{and}\quad 0=W(\omega\wedge\omega)\wedge\omega. 
\end{equation}
We have a Lie group embedding defined by
$$
\chi : \mathrm{GL}(n,\R) \to G, \quad a \mapsto \begin{pmatrix} \det a^{-1} & 0 \\ 0 & a\end{pmatrix},
$$
for $n$ even and defined by 
$$
\chi : \mathrm{GL}(n,\R) \to G, \quad a \mapsto 
\left[\begin{array}{cc} |\det a^{-1}| & 0 \\ 0 & a\end{array}\right],
$$
for $n$ odd. 

Recall that $\theta$ satisfies the equivariance property 
$$
R_g^*\theta=\mathrm{Ad}(g^{-1})\circ\theta,
$$
for all $g \in G$, where $\mathrm{Ad}$ denotes the adjoint representation of $G$. Identifying $\mathrm{GL}(n,\R)$ with its image under $\chi$, the equivariance property of $\theta$ implies that the tensor field $\eta\omega:=\eta_i\otimes\omega^i$ is invariant under the $\mathrm{GL}(n,\R)$ right action. Furthermore, since $\omega$ and $\eta$ are both semi-basic for the quotient projection $\projb \to \projb/\mathrm{GL}(n,\R)$, it follows that the smooth $2n$-manifold $M=\projb/\mathrm{GL}(n,\R)$ carries a unique signature $(n,n)$ metric $g$ and a unique non-degenerate $2$-form $\Omega$ having the property that $g$ pulls back to $\projb$ to be the symmetric part of $\eta\omega$ and $\Omega$ pulls back to $\projb$ to be the anti-symmetric part of $\eta\omega$. Moreover, we compute
$$
\aligned
0=&\;\d\left(\eta\wedge\omega\right)=\d \eta\wedge\omega-\eta\wedge\d\omega=\left[-\eta\wedge(\phi+\mathrm{Id}\tr \phi)+L(\omega\wedge\omega)\right]\wedge\omega\\
&-\eta\wedge\left[-(\phi+\mathrm{Id}\tr\phi)\wedge\omega\right]\\
=&\;L(\omega\wedge\omega)\wedge\omega,
\endaligned
$$ 
where we used~\eqref{algbianchi:again}. It follows that $\Omega$ is symplectic. 

We leave it to the interested reader to check that the pair $(g,\Omega)$ defines again a bi-La\-grang\-ian structure on $M$ whose symmetry vector fields are in one-to-one correspondence with the symmetry vector fields of $(N,[\nabla])$. Moreover, we may introduce local coordinates on $M$ so that $g$ and $\Omega$ take the form~\eqref{coordexpmetsymp}. In particular, the metric $g$ is still Einstein with non-zero scalar curvature, as can be verified by direct computation. 

\end{document}